\documentclass[11pt]{amsart}
\usepackage{graphicx,color}
\usepackage{pgfmath}
\usepackage{amsmath,amssymb,amsthm,amsfonts,multicol,ytableau}
\usepackage[vcentermath]{youngtab}
\usepackage{amscd}
\usepackage{setspace}
\usepackage{CJK}
\usepackage{halloweenmath}
\usepackage[margin=1in]{geometry}
\usepackage{comment}

\newcommand{\east}[1]{\ensuremath{\xrightarrow{\ #1\ }}}
\newcommand{\west}[1]{\ensuremath{\xleftarrow{\ #1\ }}}
\newcommand{\north}[1]{\ensuremath{\big \uparrow \!\text{\raisebox{.1ex}{\scriptsize $#1$}}}}
\newcommand{\south}[1]{\ensuremath{\big \downarrow \!\text{\raisebox{.1ex}{\scriptsize $#1$}}}}

\usepackage{todonotes}

\newtheorem{theorem}{Theorem}
\numberwithin{theorem}{section}

\newtheorem{corollary}[theorem]{Corollary}
\newtheorem{prop}[theorem]{Proposition}
\newtheorem{lemma}[theorem]{Lemma}
\newtheorem{conjecture}[theorem]{Conjecture}

\newtheorem*{maintheorem}{Theorem \ref{thm:main}}

\theoremstyle{definition}

\newtheorem{definition}[theorem]{Definition}
\newtheorem{question}[theorem]{Question}
\newtheorem{example}[theorem]{Example}

\newcommand{\defn}[1]{\textbf{#1}}

\newcommand{\ShSSYT}{\mathrm{ShSSYT}}
\newcommand{\SSYT}{\mathrm{SSYT}}

\title{Inequality of a class of near-ribbon skew Schur $Q$ functions}
\author{Maria Gillespie}
\thanks{Partially supported by NSF DMS award number 2054391.} 
\address{Department of Mathematics, Colorado State University, Fort Collins, CO, USA} \email{maria.gillespie@colostate.edu }

\author{Kyle Salois}
\address{Department of Mathematics, Colorado State University, Fort Collins, CO, USA
}
\email{kyle.salois@colostate.edu}
\date{\today}

\begin{document}

\begin{abstract}
	While equality of skew Schur functions is well understood, the problem of determining when two skew Schur $Q$ functions are equal is still largely open.  It has been studied in the case of ribbon shapes in 2008 by Barekat and van Willigenburg, and this paper approaches the problem for \textit{near-ribbon} shapes, formed by adding one box to a ribbon skew shape.  We particularly consider \textit{frayed ribbons}, that is, the near-ribbons whose shifted skew shape is not an ordinary skew shape.  We conjecture, with evidence, that all Schur $Q$ functions of frayed ribbon shape are distinct up to antipodal reflection.  We prove this conjecture for several infinite families of frayed ribbons, using a new approach via the ``lattice walks'' version of the shifted Littlewood-Richardson rule discovered in 2018 by Gillespie, Levinson, and Purbhoo.
\end{abstract}

\maketitle{}

\section{Introduction}

In this paper we provide new results on the open problem of determining when two \textit{skew Schur $Q$ functions} are equal.  The Schur $P$- and $Q$-functions $P_\lambda(x_1,x_2,\ldots)$ and $Q_\lambda(x_1,x_2,\ldots)$, are analogues of the classical Schur functions for \textit{shifted partitions} $\lambda$ (see Figure \ref{fig:shifted}), and are themselves symmetric functions that have many natural connections to representation theory and algebraic geometry. In representation theory, Schur $Q$-functions correspond to induced characters of projective representations of the symmetric group $S_n$ \cite{Stembridge}.  They also correspond to the highest weight representations of the quantum queer Lie superalgebra \cite{GJKK}, and their combinatorics may be described by crystal bases arising in this setting \cite{GHPS,GJKKK,GJKKKss}.   In algebraic geometry, products of Schur $Q$-functions govern the intersection theory of Schubert varieties in the odd orthogonal Grassmannian \cite{Pragacz}.

\begin{figure}
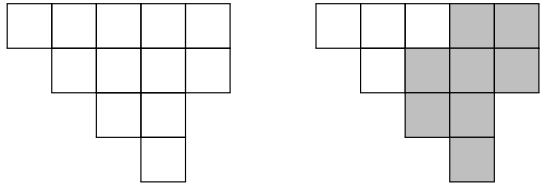

    \centering
    \begin{ytableau}
      \empty & \empty & \empty & \empty & \empty \\
      \none  & \empty & \empty & \empty & \empty \\
      \none  & \none  & \empty & \empty & \none \\
      \none  & \none  & \none  & \empty & \none 
    \end{ytableau}\hspace{1cm}
    \begin{ytableau}
      \empty & \empty & \empty & *(lightgray) \empty & *(lightgray) \empty \\
      \none  & \empty & *(lightgray) \empty & *(lightgray) \empty & *(lightgray) \empty \\
      \none  & \none  & *(lightgray) \empty & *(lightgray) \empty & \none \\
      \none  & \none  & \none  & *(lightgray) \empty & \none   
    \end{ytableau}
    \caption{At left, The shifted partition $(5,4,2,1)$, drawn in English notation.  The rows of boxes starting from the top have $5,4,2,1$ boxes in each respectively, and each successive row starts one step to the right of the previous.  At right, the gray boxes form the skew shifted shape $(5,4,2,1)/(3,1)$ formed by deleting the smaller shifted partition $(3,1)$.}
    \label{fig:shifted}
\end{figure}

Originally defined by Schur \cite{Schur}, the Schur $P$- and $Q$-functions can be defined in multiple equivalent ways, including as the $t=-1$ evaluation of the Hall-Littlewood $P$- and $Q$-polynomials \cite{Macdonald}.  Another is in terms of semistandard shifted tableaux with entries from a ``doubled" alphabet $1'<1<2'<2<3'<3<\ldots$. The latter combinatorial definition of Schur $P$- and $Q$- functions has opened opportunities for further understanding of the functions and their connections with other fields, as the theory of shifted tableaux has developed. 

Sagan \cite{sagan} and Worley \cite{Worley} developed shifted versions of combinatorial tools used in non-shifted tableaux, including the Robinson-Schensted-Knuth bijection, the Knuth equivalence relations, and jeu de taquin sliding moves. Sagan further used these to prove Stanley's conjecture that the straight-shape Schur $P$- and $Q$-functions are Schur-positive with regards to the non-shifted Schur basis $s_{\lambda}$. Building on these, Haiman \cite{Haiman} developed the process of mixed insertion on shifted tableaux, answering several more open questions in this direction. 

For Schur $Q$ functions, the above combinatorial theory gave rise to a generalized definition of Schur $Q$ functions $$Q_{\lambda/\mu}(x_1,x_2,\ldots)$$ for \textit{skew} shifted shapes $\lambda/\mu$, formed by deleting a smaller shifted partition shape $\mu$ from that of $\lambda$ (see Figure \ref{fig:shifted}).  These are known to expand positively in terms of the straight shape Schur $Q$ basis.  The resulting \textit{shifted Littlewood-Richardson coefficients} in this expansion have several known combinatorial rules \cite{Nguyen,GLP,Stembridge}, and in fact coincide with the structure coefficients that arise when multiplying two Schur $P$ functions and expressing the result back in the Schur $P$ basis.  

In the study of skew Schur $Q$ functions, the following natural problem remains largely open.

\begin{question}\label{question}
 When are two skew Schur $Q$ functions equal to each other?  More generally, when is their difference Schur $Q$ positive (when expanded in the straight shape Schur $Q$ basis)?
\end{question}

The natural analog of Question \ref{question} has been studied more thoroughly for the unshifted case of ordinary Schur functions, which similarly arise in representation theory and geometry.  In \cite{VanWilligenburg_Schur}, van Willigenburg characterized the case when a skew Schur function is equal to a straight shape Schur function, finding that $s_{\lambda/\mu}$ and $s_{\nu}$ are equal only when $\lambda/\mu$ and $\nu$ are the same shape, or $180^{\circ}$ rotations of each other. Billera, Thomas, and van Willigenburg \cite{BTvW} determined an exact condition for the equality of ribbon Schur functions.  Reiner, Shaw, and van Willigenburg \cite{ReinerEtAl} expanded on this result, giving further conditions for equality for general shapes, and soon after McNamara and van Willigenburg \cite{MvW} gave a single composition operation that maintains Schur equality.  Similar results for the problem of determining when the difference of two skew Schur functions is Schur positive were given in, for instance, \cite{Christian, KWW, MvW2, Karen}.

In the case of Schur $Q$-functions, Salmasian \cite{Salmasian} found exact criteria for when the $Q$-function $Q_{\lambda/\mu}$ of a shifted skew shape was equal to that of a shifted straight shape Schur $Q$-function $Q_{\nu}$. Barekat and van Willigenburg \cite{BarekatVanWilligenburg} investigated the problem of Schur $Q$-function equality in the case of ribbons, finding a compositional construction that gives families of shapes with equal $Q$-function, and conjecturing that it is a necessary and sufficient condition for equality. However, the remaining  results from the ordinary Schur function case have not yet been replicated for Schur $Q$-functions.

Building off of the results of Barekat and van Willigenburg \cite{BarekatVanWilligenburg}, we examine shifted skew shapes that are \textit{near-ribbons}, defined as follows.

\begin{definition}
A \defn{near-ribbon} is a connected non-ribbon shape for which it is possible to remove one square to form a ribbon.
\end{definition}

We find that this class of skew Schur $Q$-functions is self contained, in the following sense.

\begin{prop}\label{prop:near-ribbons}
If $D$ and $E$ are shifted skew shapes such that $Q_D=Q_E$, and $D$ is a near-ribbon, then $E$ is also a near-ribbon. 
\end{prop}

In this direction, we have found ample computational and theoretical evidence that, remarkably, the subclass of shifted near-ribbons that are not themselves ordinary skew shapes have \textit{distinct} Schur $Q$ functions up to antipodal reflection.  We call these shapes \textit{frayed ribbons}, and they can also be defined as follows.

\begin{definition}
A \defn{frayed ribbon} is a shifted near-ribbon containing two squares on the staircase.
\end{definition}

Two examples of frayed ribbons are shown in Figure \ref{fig:antipodal}.

We state our main conjecture precisely as follows.  Define $D^a$ to be the \textit{antipodal reflection} of a shifted skew shape across the northeast-southwest diagonal, as shown in Figure \ref{fig:antipodal}.

\begin{figure}
    \centering
    \includegraphics{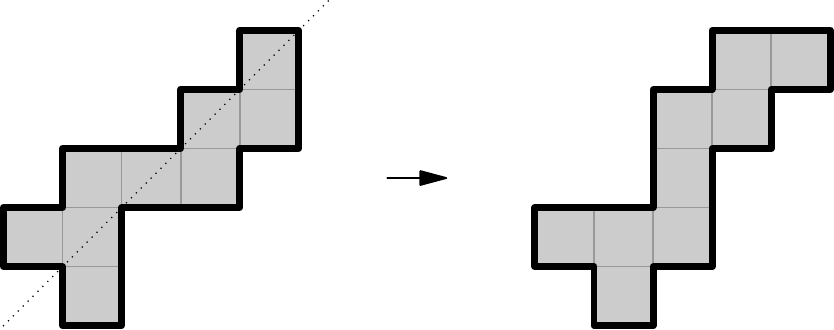}
    \caption{A frayed ribbon $D$ and its antipodal reflection $D^a$.}
    \label{fig:antipodal}
\end{figure}

\begin{conjecture}\label{conj:frayed}
If $D$ and $E$ are frayed ribbons such that $Q_D=Q_E$, then either $D=E$ or $D=E^{a}$. 
\end{conjecture}

This conjecture is in sharp contrast to the results for ribbons in \cite{BarekatVanWilligenburg}, in which infinitely many pairs of non-antipodal ribbons, formed by ``composing'' previously equal pairs in different ways, were found to have equal Schur $Q$ functions.  It is therefore surprising that adding the extra square on the staircase appears to distinguish all of the corresponding Schur $Q$ functions (up to antipodal reflection).

We have verified Conjecture \ref{conj:frayed} by computer for all frayed ribbons up to size $11$.  Towards a proof, we use the new combinatorial method based on \textit{lattice walks} developed in \cite{GLP} for computing the shifted Littlewood-Richardson coefficients, which give the decomposition of skew Schur $Q$ functions into straight-shape Schur $Q$ functions.  In particular, we use this new criterion to identify shifted semistandard Young tableaux of certain shapes with ballot reading words, and illustrate characteristics of fillings that contribute to this count. 

Using the method of lattice walks, in combination with analyzing the monomial expansions of shifted Schur $Q$ functions, we prove the following main results towards Conjecture \ref{conj:frayed}.  In the statement below, a \textit{turn} of a frayed ribbon is a square in both a nontrivial row and nontrivial column of the ribbon structure (not including the square adjacent to the two squares on the staircase).

\begin{theorem}\label{thm:main}
If $D$ and $E$ are frayed ribbons with $Q_D =Q_E$, then:
\begin{itemize}
    \item $D$ and $E$ have the same number of turns;
    \item If $D$ and $E$ have no turn or one turn, then $D=E$ or $D=E^a$;
    \item If $D$ has two turns and at most one square between the turns, then $D=E$ or $D=E^a$. 
\end{itemize}
\end{theorem}

It is worth noting that while Theorem \ref{thm:main} and computational evidence indicates that all non-antipodal pairs of distinct frayed ribbons have distinct Schur $Q$ functions, there are examples showing that several natural generalizations of Conjecture \ref{conj:frayed} do not hold.  For instance, not all non-antipodal pairs of distinct connected shifted skew shapes having at least two boxes on the staircase have distinct Schur $Q$ functions.  We provide counterexamples to this effect, and to other potential generalizations, in Section \ref{sec:conclusion}. 

The remainder of the paper is structured as follows.  In Section \ref{sec:background}, we give the necessary background for stating the results and proofs in the paper, including the combinatorial definition of skew Schur $Q$-functions and the lattice walk method for determining the shifted Littlewood-Richardson coefficients. In Section \ref{sec:near-ribbons}, we show that skew Schur $Q$-functions are distinguished by their maximal greedy fillings, and that these distinguish the Schur $Q$ functions of near-ribbon shapes from all other classes.   In Section \ref{sec:frayed-ribbons}, we prove Theorem \ref{thm:main}, and in the process we use the lattice-walk rule to completely classify the decompositions of skew Schur $Q$-functions into the straight shape Schur $Q$ basis for certain shapes.  Finally, in Section \ref{sec:conclusion}, we give examples and further observations on the distinctness, equality, or Schur $Q$ positivity of differences for related families of skew shifted shapes.

\subsection{Acknowledgments}

We thank Peter McNamara and Stephanie van Willigenburg for helpful email exchanges and conversations.  Computations in Sage \cite{sagemath} were very helpful in coming up with conjectures and examples throughout this work, and we also thank Jake Levinson for sharing some relevant Sage code.

\section{Background and notation}\label{sec:background}

In this section we outline some necessary definitions and notation for the rest of the paper.

\subsection{Ordinary and shifted skew shapes}

A \defn{partition} is a tuple $\lambda=(\lambda_1,\ldots,\lambda_k)$ of positive integers such that $\lambda_1\ge \lambda_2\ge \cdots \ge \lambda_k$.  We say the partition $\lambda$ is \defn{strict} if it is \textit{strictly decreasing}, that is, $\lambda_1>\lambda_2>\cdots>\lambda_k$.  The \defn{size} of the partition is $|\lambda|=\sum_{i}\lambda_i$, and if $n=|\lambda|$ then we say that $\lambda$ is a partition of $n$.

The \defn{Young diagram} of a partition $\lambda$ is the left-justified array of boxes in which the $i$-th row from the top contains $\lambda_i$ boxes (we use the `English' convention for Young diagrams in this paper).   If the Young diagram of $\mu$ is contained in that of $\lambda$, we write $\lambda/\mu$ for the diagram formed by the boxes appearing in $\lambda$ but not in $\mu$, and we call such a diagram an \defn{(ordinary) skew shape}.  A shape given by a single partition diagram $\lambda$ is called a \defn{straight shape} (though straight shapes are also considered skew, with $\mu$ being the empty partition).

For strict partitions, it is common to ``shift'' their Young diagrams as follows.

\begin{definition}
  The \defn{shifted Young diagram} of a strict partition $\lambda$ is formed by taking the ordinary Young diagram of $\lambda$ and shifting the $i$-th row from the top $i-1$ steps to the right for all $i$.  A \defn{shifted skew shape} is the difference $\lambda/\mu$ of two nested shifted diagrams.
\end{definition}

We can think of shifted diagrams as fitting inside a triangular ``staircase shape''.  A square in a skew shifted diagram is said to be on the \defn{staircase} if it is in the leftmost possible position in its row (see Figure \ref{fig:diagrams}).

\begin{figure}
    \centering
    \includegraphics{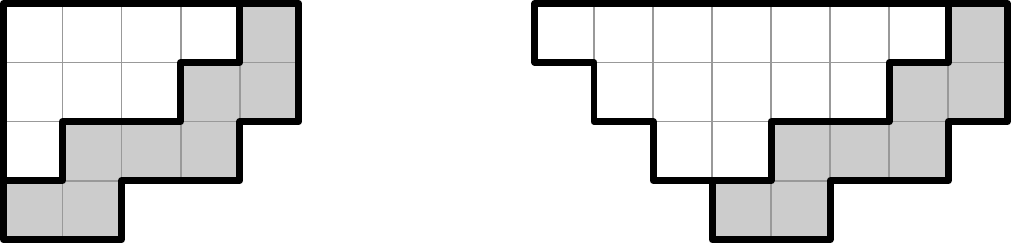}
    \caption{The ordinary skew shape $(5,5,4,2)/(4,3,1)$ at left, and the shifted skew shape $(8,7,5,2)/(7,5,2)$ at right.  Both sets of shaded squares have the same underlying shape, but their indexing is different since the latter is shifted.  In the latter, one box is on the staircase.}
    \label{fig:diagrams}
\end{figure}

A \defn{ribbon} is a connected ordinary skew shape that does not contain a $2\times 2$ box of squares.  The skew diagrams in Figure \ref{fig:diagrams} are both considered ribbons, since even though the latter is a shifted skew shape, its underlying diagram is equivalent to a ribbon.  A \defn{near-ribbon} is a connected non-ribbon shape for which it is possible to remove one square to form a ribbon.   A \defn{frayed ribbon} is a shifted near-ribbon containing two squares on the staircase.

Notice that frayed ribbons are precisely the shifted near-ribbons that are not ordinary unshifted skew shapes.  Two examples are given in Figure \ref{fig:near-ribbons}.

\begin{figure}[t]
    \centering
    \includegraphics{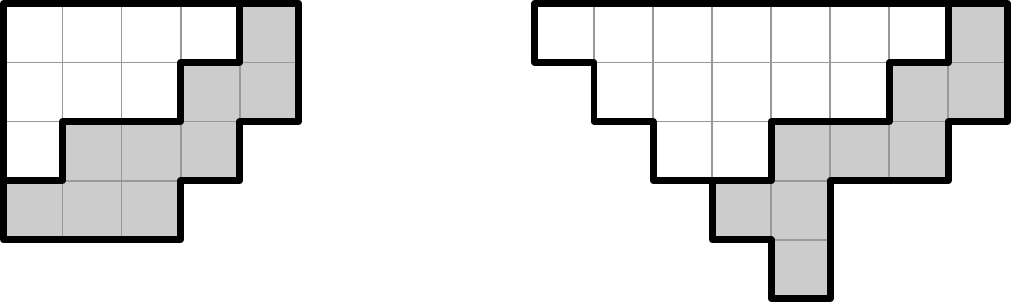}
    \caption{Two near-ribbon shapes, the latter of which is frayed.}
    \label{fig:near-ribbons}
\end{figure}

\subsection{Shifted tableaux and Schur $Q$ functions}

A \defn{(ordinary) semistandard Young tableau}, or SSYT, is a way of filling the boxes of an ordinary skew shape with positive integers such that the rows are weakly increasing from left to right and the columns are strictly increasing from top to bottom.  The analogous definition for shifted shapes makes use of the ``doubled alphabet'' of symbols: \begin{equation}\label{eq:alphabet}
    1'<1<2'<2<3'<3<\cdots.
\end{equation}

\begin{definition}
A \defn{shifted semistandard Young tableau} (ShSSYT) is a filling of a shifted skew shape with primed and unprimed letters from (\ref{eq:alphabet}) such that rows and columns are weakly increasing from left to right and top to bottom, primed letters can only be repeated in columns, and unprimed letters can only be repeated in rows.

We write $\ShSSYT(\lambda/\mu)$ for the set of all shifted semistandard Young tableaux of shape $\lambda/\mu$.  The \defn{reading word} of a shifted tableau is the word formed by concatenating the rows from bottom to top, and the \defn{reading order} of the entries in a tableau is the total order given by the reading word. An example is shown in Figure \ref{fig:ribbon_plus_box}.
\end{definition}

\begin{figure}[b]
    \centering
    
    $$ \begin{ytableau} \none & 1' & 1 & 1 & 1 \\
     \none & 1' & 2 \\ 1' & 1 \\ 1 \end{ytableau} $$
    
    \caption{A shifted semistandard Young tableau of a near-ribbon shape, with reading word $11'11'21'111$. Its monomial is $x_1^8x_2$.  Observe that the first $1^\ast$ and $2^\ast$ in reading order can be changed to either $1'$ or $1$, or $2'$ or $2$, respectively, while maintaining that the filling is a ShSSYT.}
    \label{fig:ribbon_plus_box}
\end{figure}

\begin{definition}[Starred entries] We write $1^\ast$ to denote a letter that is either $1'$ or $1$, $2^\ast$ to denote a letter that is either $2'$ or $2$, etc.
\end{definition}  
The \defn{monomial} associated to a shifted semistandard Young tableau $T$ is $$x^T:=x_1^{m_1}x_2^{m_2}\cdots$$ where $m_i$ is the total number of $i^\ast$ entries in $T$ for each $i$.  The tuple $(m_1,m_2,m_3,\ldots)$ of exponents is called the \defn{content} of $T$. We often write $X$ for the set of variables $x_1,x_2,\ldots$.

\begin{definition}
The \defn{skew Schur $Q$ function} for the skew shape $\lambda/\mu$ is the symmetric function $$Q_{\lambda/\mu}(X)=\sum_{T\in \ShSSYT(\lambda/\mu)}x^T.$$
\end{definition}

Schur $Q$ functions are the natural shifted analog of ordinary Schur functions $s_{\lambda/\mu}$, defined as $$s_{\lambda/\mu}(X)=\sum_{T\in \SSYT(\lambda/\mu)}x^T,$$ where $\SSYT(\lambda/\mu)$ is the set of all ordinary SSYT's of shape $\lambda/\mu$, and $x^T=x_1^{m_1}x_{2}^{m_2}\cdots$ where $m_i$ is the number of times $i$ appears in $T$.

\begin{definition}
The \defn{Schur $P$ function} for a straight shape $\lambda$ is the symmetric function $$P_{\lambda}(X)=2^{-\ell(\lambda)}Q_\lambda(X)$$ where $\ell(\lambda)$ is the number of parts of $\lambda$.  
\end{definition}
Since the first letter of each row of a straight shape ShSSYT may be either primed or unprimed, dividing by $2^{\ell(\lambda)}$ in the definition above may also be interpreted as summing over all ShSSYT's of shape $\lambda$ in which the first letter of each row is unprimed. 

The Schur $Q$ and $P$ functions are both known to be elements of the subring $\Omega$ of the ring of symmetric functions $\Lambda$, which are defined as follows.  (Throughout, we take $\mathbb{Q}$ as the field of coefficients over which our symmetric functions are defined.)

\begin{definition}
 We write $\Lambda(X)$ for the ring of symmetric functions over $\mathbb{Q}$ in the variables $X=\{x_1,x_2,x_3,\ldots\}$.  We also write $X_n=\{x_1,\ldots,x_n\}$ and $\Lambda(X_n)$ for the ring of symmetric polynomials in the finite set of variables $X_n$.
\end{definition}

\begin{definition}
 Let $p_i(X)=x_1^i+x_2^i+\cdots\in \Lambda(X)$ be the $i$-th power sum symmetric function. Then $\Omega(X)$ is the subring of $\Lambda(X)$ generated as an algebra by the odd degree power sums $p_1(X),p_3(X),p_5(X),\ldots$.  
\end{definition}

The Schur $P$ and $Q$ functions $P_\lambda$ and $Q_\lambda$ each separately form a basis of $\Omega(X)$ as a vector space, where $\lambda$ ranges over all straight shape strict partitions.  They also form dual bases with respect to the natural analog of the Hall inner product.  (See \cite{Macdonald}).

\subsection{Antipodal reflections}

In \cite{BarekatVanWilligenburg}, a group of four transformations on \textit{ordinary} skew shapes that preserves the Schur $Q$ function was identified.  The group, isomorphic to $\mathbb{Z}/2\mathbb{Z}\times \mathbb{Z}/2\mathbb{Z}$, is generated by two involutions: ordinary  \textit{transposition} of Young diagrams (i.e., reflecting the skew diagram about the northwest-southeast diagonal), and the \textit{antipodal} map defined below.

\begin{definition}
Given a shifted skew shape $D$, the \defn{antipodal reflection} $D^a$ is the shape formed by reflecting $D$ across the northeast-southwest diagonal.
\end{definition}

The other operation of transposition is only well-defined on shifted skew shapes that happen to also be ordinary skew shapes (as a set of squares), which was sufficient in the paper \cite{BarekatVanWilligenburg} that considered ribbon shapes, since all ribbons are both shifted and ordinary skew shapes.  Here, however, since we are considering frayed ribbons, we will be primarily be using the antipodal map.

In \cite{DeWitt}, it was shown that the antipodal map sends shifted skew shapes to shifted skew shapes and preserves the Schur $Q$ functions.  We state this result here as we will be referring to it frequently throughout the paper.

\begin{prop}[\cite{DeWitt}, Theorem IV.13]\label{prop:antipodal}
Let $D$ be a shifted skew shape. Then $Q_{D} = Q_{D^a}$.
\end{prop}

\subsection{Walks, ballot tableaux, and shifted Littlewood-Richardson coefficients}

Any product of Schur $P$ functions can be expressed in terms of the Schur $P$ basis, and any skew Schur $Q$ function can be expressed in terms of the Schur $Q$ basis in $\Omega$.  The shifted Littlewood-Richardson coefficents give a formula for both of these expansions, which are related by the duality between Schur $Q$ and $P$ functions.  In particular, there exist nonnegative integers $f_{\mu\nu}^\lambda$ for each triple of strict partitions $\lambda,\mu,\nu$ with $|\mu|+|\nu|=|\lambda|$, such that
$$P_\mu P_\nu =\sum_{\lambda} f^{\lambda}_{\mu\nu} P_\lambda\hspace{1cm}\text{and}\hspace{1cm} Q_{\lambda/\mu}=\sum_{\nu} f^{\lambda}_{\mu\nu} Q_\nu.$$ 

The coefficients $f^{\lambda}_{\mu\nu}$ are known as the \defn{shifted Littlewood-Richardson coefficients}.  While several combinatorial interpretations of these coefficients are known \cite{Nguyen,Stembridge}, in this paper we will be primarily using the one established in \cite{GLP} via lattice walks.  We recall some definitions from \cite{GLP}.

\begin{definition}
Let $w$ be a word in the alphabet $\{1',1,2',2\}$.  The \defn{$1/2$-walk} of the word $w$ is a lattice walk in the first quadrant using one of the four unit steps \[
\east{~~} \ =\  (1,0)
\qquad \west{~~} \ =\  (-1,0)
\qquad \north{} \ =\  (0,1)
\qquad \south{} \ =\  (0,-1)\,.
\] for each letter of the word as we read from left to right. The walk starts at the origin $(0,0)$, and at the $i$-th step we read $w_i$ and draw the next step of the walk according to
Figure \ref{fig:lattice-walk}, with two cases based on whether or not the step starts on one of the $x$ or $y$ axes.  In particular, any $2$ is an up arrow, any $1'$ is a right arrow, a $1$ is either right if on an axis or down if not, and a $2'$ is either up if on an axis or left if not.  We will generally write the label each step of the walk by the letter $w_i$, so as to represent both the word and its walk on the same diagram.
\end{definition}

\begin{figure}
	\begin{center}
		\includegraphics{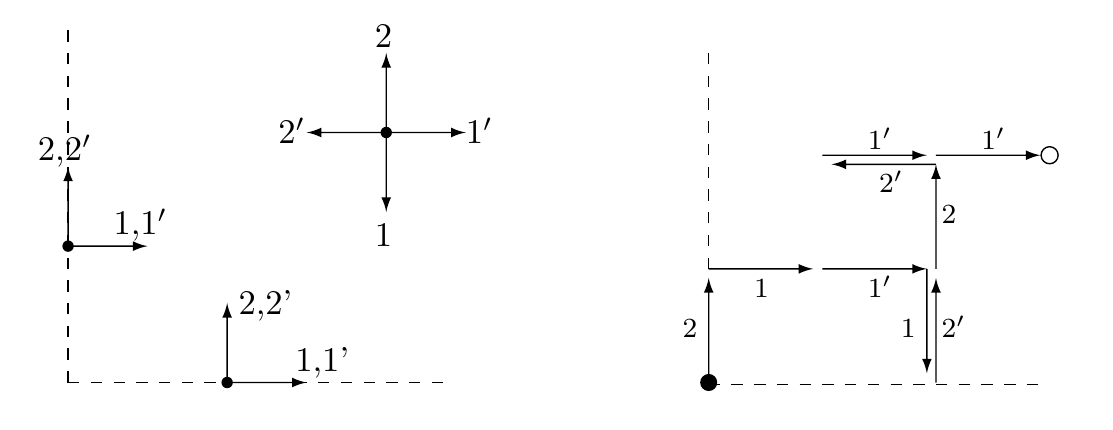}
		\end{center}
\caption{At left, the directions assigned to the letters $1',1,2',2$ in the lattice walk of a word, depending on whether or not the step starts on an axis. At right, the walk for $w=211'12'22'1'1'$.\label{fig:lattice-walk}} \end{figure}

If $w$ is a word in the alphabet $\{i',i,(i+1)',i+1\}$ for any $i$, we similarly define the \defn{$i/(i+1)$-walk} of $w$ by replacing $1^\ast$ with $i^\ast$ and $2^\ast$ with $(i+1)^\ast$ in the above definition.  We can then define the $i/(i+1)$-walk of any word in the doubled alphabet as follows.

\begin{definition}
If $w$ is a word in the alphabet $\{1',1,2',2,3',3,\ldots\}$, the \defn{$i/(i+1)$-walk} of $w$ is the $i/(i+1)$-walk of the subword of $w$ formed by its $i^\ast$ and $(i+1)^\ast$ elements (which we often refer to as the \textbf{$i/(i+1)$-subword}).
\end{definition}

\begin{definition}
A word $w$ is \defn{ballot} if, for every $i$, the $i/(i+1)$-walk of $w$ ends at a point on the $x$ axis.
\end{definition}

This notion of ballotness is a analog of the ``Yamanouchi'' condition for ordinary Littlewood-Richardson coefficients.

\begin{example}\label{ex:ballot}
Suppose $w=212'231'3'1'121'11$.  Then its $1/2$-subword is $212'21'1'121'11$, whose walk is shown below at left.  Its $2/3$-subword is $22'233'2$, whose walk is shown below at right.  Both walks end on the $x$-axis, so the word $w$ is ballot.
\begin{center}
\includegraphics{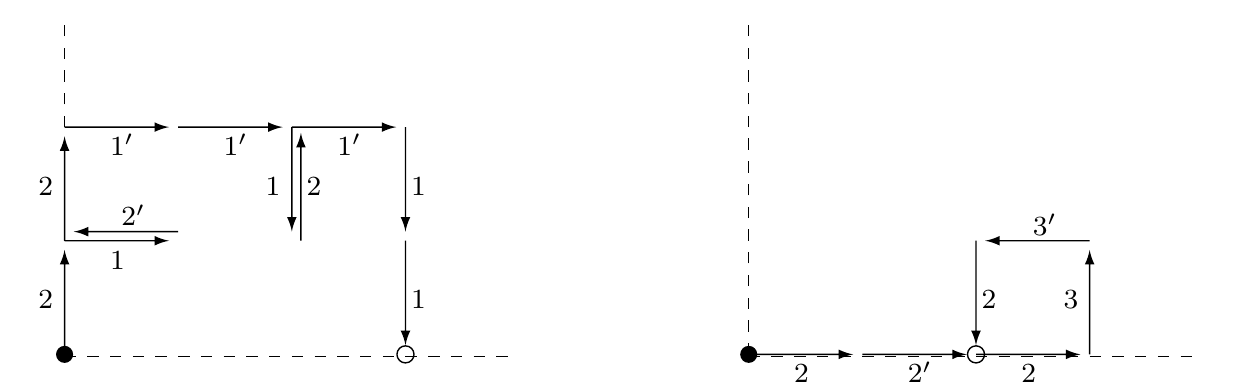}
\end{center}
\end{example}

We now need to recall the definition of a tableau in \textit{canonical form}.

\begin{definition}
A skew shifted semistandard Young tableau is in \defn{canonical form} if the first $i^\ast$ in reading order is unprimed for every $i$.
\end{definition}

Note that if $T$ is semistandard, its first $i^\ast$ in reading order may be changed to being primed or unprimed and always still yield another semistandard tableau.  Thus, the canonical form is simply enforcing this choice to be unprimed.  The shifted jeu de taquin process \cite{sagan,Worley} is only well-defined for skew tableaux in canonical form, which is why it appears in the shifted Littlewood-Richardson rule.

We finally define a \textit{ballot tableau} as follows.

\begin{definition}
A skew shifted Young tableau $T$ is a (shifted) \defn{ballot} tableau if it is semistandard, in canonical form, and its reading word is ballot.
\end{definition}

\begin{example}
The tableau 
$$\begin{ytableau}
  \none & \none & \none & 1' & 1 & 1 \\
  \none & \none & 1' & 1 & 2 \\ 
  \none & \none & 1' & 3' \\
  1 & 2' & 2 & 3 \\
  2
\end{ytableau}$$
is a ballot tableau.  Note that it is in canonical form; the first $1$, the first $2$, and the first $3$ in reading order are all unprimed.  Its reading word is $212'231'3'1'121'11$, which is the ballot word from Example \ref{ex:ballot}.
\end{example}

Shifted ballot tableaux are also sometimes referred to as \textit{shifted Littlewood-Richardson tableaux}, because they enumerate the shifted Littlewood-Richardson coefficients.  We may now state the shifted Littlewood-Richardson rule.

\begin{theorem}[\cite{GLP}, Theorem 1.5]  The shifted  Littlewood-Richardson coefficient $f^\lambda_{\mu\nu}$ is equal to the number of shifted ballot tableaux of shape $\lambda/\mu$ and content $\nu$.
\end{theorem}

\section{Preliminaries on near-ribbons and general shifted shapes}\label{sec:near-ribbons}

We first show how to combinatorially distinguish  Schur $Q$ functions from each other via the leading monomial in lexicographic order.

\begin{definition}
For a shifted skew shape $\lambda/\mu$, define its \defn{greedy filling} to be the labeling formed by:
\begin{itemize}
    \item First placing $1$s in the maximal `inner strip' of ribbons, by placing them in every square that is either in the top row or that shares a corner or edge with the inner shape $\mu$,
    \item Then placing $2$s maximally to form an inner strip of ribbons in the remaining empty squares, namely, in every square sharing a corner or edge with one of the $1$s,
    \item Then placing $3$s maximally in the same way in the remaining squares, and so on.
\end{itemize}
(See Figure \ref{fig:greedy}.) The \defn{greedy monomial} of the shape is the monomial $$2^r x_1^{m_1} x_2^{m_2}\cdots$$ where $r$ is the total number of maximal, connected, uniformly-labeled ribbons formed by the greedy filling, and $m_i$ is the number of squares labeled by $i$ for each $i$.
\end{definition}

\begin{figure}
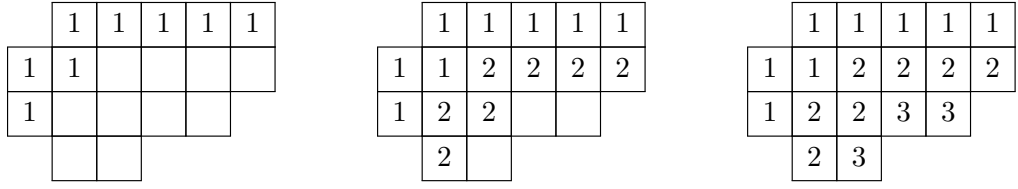

    \centering
    \begin{ytableau} \none & \none & 1 & 1 & 1 & 1 & 1 \\  \none & 1 & 1 & \empty & \empty & \empty & \empty \\  \none & 1 & \empty & \empty & \empty & \empty \\  \none & \none & \empty & \empty  \end{ytableau} \hspace{0.5cm} \begin{ytableau}  \none & \none & 1 & 1 & 1 & 1 & 1 \\  \none & 1 & 1 & 2 & 2 & 2 & 2 \\ \none & 1 & 2 & 2 & \empty & \empty \\ \none & \none & 2 & \empty  \end{ytableau} \hspace{0.5cm} \begin{ytableau} \none & \none & 1 & 1 & 1 & 1 & 1 \\ \none & 1 & 1 & 2 & 2 & 2 & 2 \\ \none & 1 & 2 & 2 & 3 & 3 \\ \none & \none & 2 & 3 \end{ytableau}
    \caption{Forming the greedy filling of a shifted skew shape.  The greedy monomial is $2^4 x_1^8x_2^7x_3^3$.}
    \label{fig:greedy}
\end{figure}

\begin{prop}\label{prop:greedy}
Suppose shifted skew shapes $D$ and $E$ have different greedy monomials.  Then $Q_D\neq Q_E$.
\end{prop}

\begin{proof}
  Consider the lexicographic ordering of monomials (based on their sequence of exponents in the order $x_1,x_2,x_3,\ldots$).  We claim that the leading term of $Q_D$ with respect to this ordering is the same as the greedy monomial of $D$.
  
  Note that for any label $i$ in the greedy filling and any connected component $R$ of the corresponding disjoint union of ribbons, there are exactly two ways to prime or unprime each of the letters $i$ in the ribbon $R$ so as to make the ribbon semistandard.  In particular, the lower left entry of the ribbon may be either primed or unprimed, and the rest of the entries are determined.  An example corresponding to Figure \ref{fig:greedy} is shown:
  $$ \begin{ytableau} \none & \none & 1' & 1 & 1 & 1 & 1 \\ \none & 1' & 1 & 2' & 2 & 2 & 2 \\ \none & 1^\ast & 2' & 2 & 3^\ast & 3 \\ \none & \none & 2^\ast & 3^\ast \end{ytableau} $$
    where we recall that the $\ast$'s indicate the numbers that may be either primed or unprimed.  This forms a shifted semistandard Young tableau whose contribution to the Schur $Q$ function is the monomial $x_1^{m_1}x_2^{m_2}\cdots $, and since there are two choices for every connected ribbon in the decomposition, this monomial has a coefficient of $2^r$ in $Q_D$. This completes the proof.
\end{proof}

As a corollary, we find that near-ribbon shapes have distinct Schur $Q$ functions from all other types of shapes. 

\begin{corollary}[Proposition \ref{prop:near-ribbons}]
Suppose that $D$ is a near-ribbon. Then if $Q_D=Q_E$ for some shifted skew shape $E$, $E$ must also be a near-ribbon. 
\end{corollary}

\begin{proof}
  If $D$ has size $n$, since it is a near-ribbon its greedy monomial is $4x_1^{n-1}x_2$. Since $Q_D=Q_E$, it follows from Proposition \ref{prop:greedy} that  $4x_1^{n-1}x_2$ is the greedy monomial of $E$, so $E$ must decompose into one connected ribbon of $1$s and a single square containing a $2$ (since the coefficient of $4=2^2$ means there are two connected ribbons formed by the greedy filling).  Therefore shape $E$ can be formed by adding a box to a ribbon, which is either a ribbon itself, a ribbon plus a disconnected box, or a near-ribbon.  In the former two cases, the greedy monomial for $E$ is in fact $cx_1^n$ for some $c$, a contradiction.  Therefore $E$ is a near-ribbon.
\end{proof}

\section{Distinctness for frayed ribbon shapes}\label{sec:frayed-ribbons}

We now turn our attention to Conjecture \ref{conj:frayed} and proving the four statements of Theorem \ref{thm:main}.

\begin{definition}
A \defn{turn} in a frayed ribbon is a sub-diagram of the following form $$\begin{ytableau} \none & \empty \\ \empty & \empty  \end{ytableau} \hspace{0.5cm} \textrm{ or } \hspace{0.5cm} \begin{ytableau} \empty & \empty \\ \empty & \none \end{ytableau}$$
that does not contain either of the two boxes on the staircase.  We call these two types of turns \defn{outer turns} and \defn{inner turns} respectively.
\end{definition}

We now restate Theorem \ref{thm:main} here for the reader's convenience.

\begin{maintheorem}
If $D$ and $E$ are frayed ribbons with $Q_D =Q_E$, then:
\begin{itemize}
    \item $D$ and $E$ have the same number of turns;
    \item If $D$ and $E$ have no turn or one turn, then $D=E$ or $D=E^a$;
    \item If $D$ has two turns and at most one square between the turns, then $D=E$ or $D=E^a$.
\end{itemize}
\end{maintheorem}

In the sections below, the following lemma will come in handy.

\begin{lemma}\label{lem:top-row}
The top row of any Littlewood-Richardson tableau has only $1^\ast$ entries, and in fact has at most one $1'$.
\end{lemma}

\begin{proof}
  By the ballot condition on the $1/2$-walk, the last $1^\ast$ or $2^\ast$ in the word must be a $1^\ast$, since no $2^\ast$ arrow can ever end on the $x$ axis.  Similarly the last $2^\ast$ in any ballot reading word comes after the last $3^\ast$, and so on, meaning that the last letter in reading order is $1^\ast$.  Thus, by the semistandard condition, the entire top row consists of $1^\ast$ entries, with at most one $1'$ at the start of the row.
\end{proof}

\subsection{Distinguishing based on number of turns}

We begin by proving the first statement of Theorem \ref{thm:main}.

\begin{prop}\label{prop:distinguish_turns}
Suppose $D$ is a frayed ribbon shape of size $n$ with $k$ turns. Then the coefficient of $Q_{(n-2,2)}$ in the expansion of $Q_D$ is $2k$. 
\end{prop}

\begin{proof}
We assume without loss of generality, by Proposition \ref{prop:antipodal}, that the second-to-bottom row of $D$ has more than two squares.  First, note that in a ballot tableau of shape $D$ with content $(n-2,2)$, the entry in the bottom row cannot be $1^\ast$ by semistandardness, and cannot be $2'$ by the canonical form condition, so it must be a $2$.   Moreover, the other $2^\ast$ must be either in the corner of one of the outer turns or at the end of the top row, or else there would be a $1^\ast$ to its right or below it, contradicting semistandardness.  But the $2^\ast$ cannot be at the end of the top row by Lemma \ref{lem:top-row}.  Thus the second $2^\ast$ must be in the corner of an outer turn.

 Now, if $t_0$ is the number of outer turns and $t_1$ is the number of inner turns, we have $t_0+t_1=k$ and either $t_0=t_1$ (if the topmost turn is an inner turn) or $t_0=t_1+1$ (if the topmost turn is an outer turn).  See Figure \ref{fig:turns} for each such case.

\begin{figure}
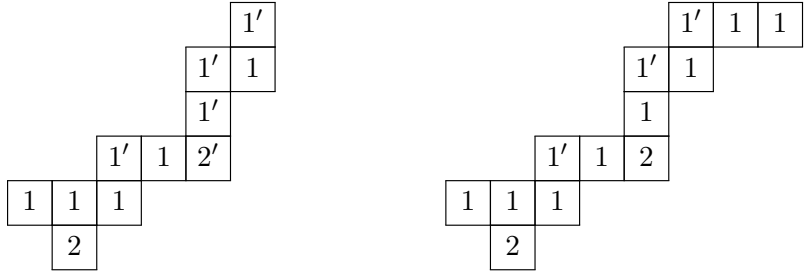

    \centering
    \begin{ytableau}
       \none & \none & \none & \none & \none & 1' \\
       \none & \none & \none & \none & 1' & 1 \\
       \none & \none & \none & \none & 1' \\
       \none & \none & 1' & 1 & 2' \\
       1 & 1 & 1 \\
       \none & 2
    \end{ytableau} \hspace{2cm}
    \begin{ytableau}
       \none & \none & \none & \none & \none & 1' & 1 & 1 \\
       \none & \none & \none & \none & 1' & 1 \\
       \none & \none & \none & \none & 1 \\
       \none & \none & 1' & 1 & 2 \\
       1 & 1 & 1 \\
       \none & 2
    \end{ytableau}
    \caption{At left, a ballot tableau of a frayed ribbon with three outer turns and two inner turns, that is, $t_0=3$ and $t_1=2$ in the notation of the proof of Proposition \ref{prop:distinguish_turns}.  At right, a ballot tableau of a frayed ribbon with $t_0=t_1=3$.}
    \label{fig:turns}
\end{figure}

For each outer turn of $D$, there are exactly four semistandard fillings containing a $2^\ast$ in the corner of that turn; in particular, the $2^\ast$ may be either $2$ or $2'$, and the square above it may be either $1$ or $1'$.  We now check which of these are ballot.  The reading word is of the form $$211(1^\ast \cdots 1^\ast)2^\ast1^\ast (1^\ast\cdots 1^\ast) $$ where the strings $(1^\ast\cdots 1^\ast)$ in parentheses have a mix of primed and unprimed $1$ entries that are uniquely determined by the shape of $D$.  Just before the second $2^\ast$, the walk is on the $x$ axis, and the $2^\ast$ lifts it to a position just above the $x$ axis.  In order for the walk to return to the $x$ axis in the end, it is necessary and sufficient that an unprimed $1$ appears after the $2^\ast$, as shown in the walk below (which corresponds to the tableau at left in Figure \ref{fig:turns}). 
\begin{center}
\includegraphics{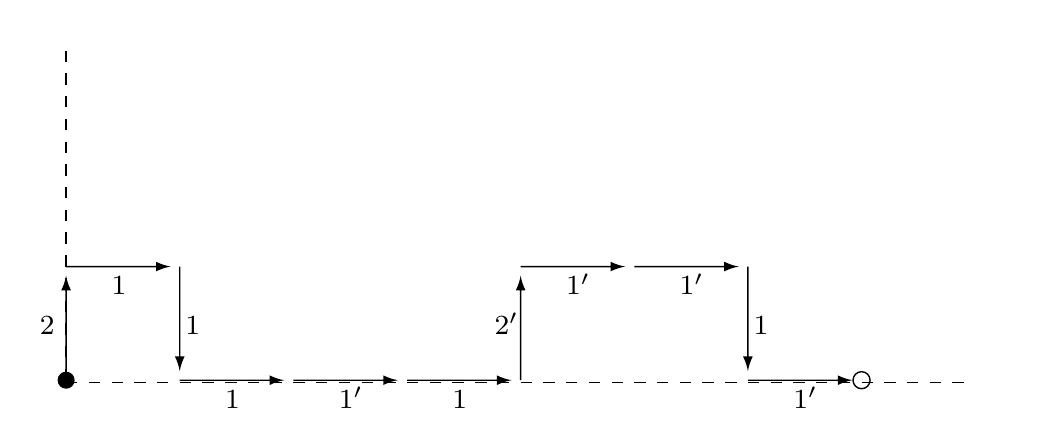}
\end{center}
This is guaranteed to happen by semistandardness if there is an inner turn after the outer turn containing $2^\ast$, but if not, then the only way it is guaranteed is if the $1^\ast$ just following the $2^\ast$ is unprimed (since all other $1^\ast$s in the final column must be primed by semistandardness).

Thus, if $t_0=t_1$ then we have an inner turn after all outer turns, and so each of the $t_0$ outer corners contributes four ballot tableaux.  So the coefficient is $$4t_0=2(t_0+t_1)=2k.$$  If instead $t_0=t_1+1$, then the $t_0-1$ lowest outer corners contribute four ballot tableaux, but the topmost outer corner contributes only two since the $1^\ast$ above it must be unprimed.  Thus we have a coefficient of $$4(t_0-1)+2=2t_0+2t_0-2=2t_0+2(t_1+1)-2=2(t_0+t_1)=2k.$$  Therefore, in all cases, the coefficient is $2k$ as desired.
\end{proof}

As a corollary we obtain the first statement of Theorem \ref{thm:main}.

\begin{corollary}\label{cor:turns}
Let $D$ and $E$ be frayed ribbon shapes for which $Q_D=Q_E$. Then $D$ and $E$ have the same number of turns.
\end{corollary}

\subsection{Frayed ribbons with one turn}

We now focus on the second statement of Theorem \ref{thm:main}.  Note that there are only two frayed ribbon shapes with no turns -- the straight shifted shape $(n-1,1)$ and its antipodal reflection -- so the statement holds when there are no turns.  

We therefore consider the case of one turn. We begin by proving several lemmas about the structure of all ballot tableaux of one-turn shapes.  Throughout this entire section we let $T$ be a ballot tableau of a frayed ribbon shape with one turn.    Since we are considering antipodal pairs, we also assume without loss of generality that the shape of $T$'s frayed ribbon goes ``to the right and then up'' from the frayed part, that is, its second-to-bottom row has more than two squares.  

\begin{lemma}\label{L_Flap_Column}
In $T$, let $k^\ast$ be the largest entry of the rightmost column.  Then for all $1\le i<k$, at least one $i^\ast$ appears in the column, including exactly one unprimed $i$. 
\end{lemma}

\begin{figure}
    \centering
     \begin{ytableau}     
    \none & \none & \none & \none & \none & \none & 1 \\
    \none & \none & \none & \none & \none & \none & 2'  \\   \none & \none & \none & \none & \none & \none & 2' \\   \none & \none & \none & \none & \none & \none & 2  \\ 
    \none & \none & \none & \none & \none & \none & 3  \\
    \none & \none & \none & \none & \none & \none & 4' \\ 
    1 & 1 & 1 & 1 & 1 & 1 & 4 \\ \none & 2      \end{ytableau} 
    \hspace{2cm}
    \begin{ytableau}     
    \none & \none & \none & \none & \none & \none & 1' \\
    \none & \none & \none & \none & \none & \none & 1  \\   \none & \none & \none & \none & \none & \none & 2' \\   \none & \none & \none & \none & \none & \none & 2'  \\ 
    \none & \none & \none & \none & \none & \none & 2'  \\
    \none & \none & \none & \none & \none & \none & 2 \\ 
    1 & 1 & 1 & 1 & 1 & 1 & 3 \\ \none & 2      \end{ytableau}
    \caption{At left, an example of a tableau $T$ that satisfies the conditions of statements of Lemmas \ref{L_Flap_Column}, \ref{L_Flap_Row}, and \ref{L_Flap_Box} At right, a tableau $T$ that also satisfies Lemmas \ref{L_Flap_Corner}, \ref{L_FLap_OneTwo}, and \ref{lem:where-3}.}
    \label{fig:rf_oneturn_ex}.  
\end{figure}

\begin{proof}
First, since $T$ is semistandard, the entries of the column are weakly increasing from top to bottom.  In particular, the $k^\ast$ entries in this column form a consecutive string of the reading word, then the $(k-1)^\ast$ entries are consecutive after them, and so on.

Since $T$ is ballot, the $(k-1)/k$-walk returns to the $x$-axis.  Just after reading the string of $k^\ast$'s in the column, the lattice walk cannot be on the $x$-axis (since it never can be after a $k^\ast$ step), and the only arrow that can move the walk downwards is an unprimed $k-1$.  Thus there is an unprimed $k-1$ in the column.  This holds for all pairs $i-1$ and $i$ for $1<i\leq k$, so the column must contain nonempty strings of each $i$ for $1\leq i\leq k$, with an unprimed entry guaranteed for $1\leq i<k$. Since $T$ is semistandard, there is at most one unprimed $i$ in the column for each $i$ as well (with all the $i'$ entries occurring above it).
\end{proof}

\begin{lemma}\label{L_Flap_Row}
In $T$, the long row has entries $1,\ldots,1,k^\ast$ for some $k$.
\end{lemma}

\begin{proof}
As in the previous lemma, let $k^\ast$ be the entry in the bottom right corner. First, assume for contradiction that one of the entries, other than the rightmost entry, is $i^\ast$ for $1<i< k$. Then the $(i-1)/i$-walk reaches a point above the $x$-axis just after this $i^\ast$. Since $T$ is semistandard, there is no $(i-1)$ in the reading word until after the unprimed $i$ in the column guaranteed by Lemma~\ref{L_Flap_Column}. After the unprimed $i$, the walk is at least two steps above the $x$-axis, but by Lemma~\ref{L_Flap_Column}, there is only one more instance of an unprimed $i-1$, so the walk does not return to the $x$-axis. This contradicts that $T$ was ballot, and thus every entry in the long row is either $1^\ast$ or $k^\ast$.

If some entry other than the bottom right entry is a $k^\ast$, then since $T$ is semistandard, the bottom right entry is an unprimed $k$. Then, after the bottom right entry of $k$, the $(k-1)/k$-walk is at least two steps above the $x$-axis, and there is only one unprimed $k-1$, and thus one step down, in the remainder of the walk. Therefore the lattice walk for the subword of $(k-1)^\ast$ and $k^\ast$ entries does not return to the $x$-axis, contradicting the ballotness of $T$. 

Thus every entry in the long row, other than the bottom right entry, is $1^\ast$. Since $T$ is in canonical form, the first $1$ must be unprimed, and so by semistandardness the long row has entries $1,\ldots,1,k^\ast$.
\end{proof}

\begin{lemma}\label{L_Flap_Box}
In $T$, the entry in the square in the bottom row is a $2$.
\end{lemma}

\begin{proof}
Since $T$ is ballot and hence in canonical form, any entry in the square in the bottom row must be unprimed, since it is first in reading order.  If the entry were a $1$, then there would be a vertical adjacency of two unprimed $1$s by Lemma \ref{L_Flap_Row}, contradicting that $T$ is semistandard.  

Suppose that the entry in the bottom row is $i$ for $i>2$.  Then in the lattice walk for the subword of $(i-1)^\ast$ and $i^\ast$ entries, by Lemma~\ref{L_Flap_Row}, there is no $(i-1)^\ast$ entry between the entry in the bottom row and the $i^\ast$ entries in the column, so the lattice walk begins with two up steps. Then by Lemma~\ref{L_Flap_Column}, there is at most one more down step in the lattice walk from the single unprimed $i-1$, and so the walk does not return to the $x$-axis, contradicting that $T$ is ballot. Thus the entry in the bottom row is a 2.
\end{proof}

See the first diagram in Figure \ref{fig:rf_oneturn_ex} for an example of a tableau that satisfies the three lemma statements above.

\begin{lemma}\label{L_Flap_Corner}
In $T$, the entry in the corner of the outer turn is at most 3. 
\end{lemma}

\begin{proof}
Suppose the entry in the lower right corner is $i^\ast$ for $i>3$, and consider the $(i-1)/i$-walk. Since $i>3$, there is no instance of $(i-1)^\ast$ or $i^\ast$ in the word before the entry from the lower right corner by Lemmas \ref{L_Flap_Row} and \ref{L_Flap_Box}. Since $T$ is ballot, the first $i$ is unprimed, so the lower right entry is $i$. Then the $(i-1)/i$-subword has the form $i,i',\ldots,i',i-1,i-1',\ldots,i-1'$, with $k$ copies of $i^\ast$ for some $k$. However, after the string of $i^\ast$s, the walk has taken $k$ steps up the $y$-axis, followed by a right step for the $i-1$ entry, and right steps for the remaining $i-1'$ entries. Thus the walk does not return to the $x$-axis, contradicting that $T$ is ballot, and hence the lower right entry is at most $3$. 
\end{proof}

\begin{lemma}\label{L_FLap_OneTwo}
In $T$, there are at least as many $1$s in the long row as there are $2^\ast$s in $T$.
\end{lemma}

\begin{proof}
Consider the subword of $T$ containing $1^\ast$s and $2^\ast$s. From Lemma~\ref{L_Flap_Box}, we have that the bottommost square contains a 2, and from Lemma~\ref{L_Flap_Row}, the remaining $2^\ast$ entries are in the rightmost column. Suppose $T$ contains more $2^\ast$ entries than $1^\ast$ entries in the second row, and say the number of $1^\ast$ entries in the second row is $j$. Then, following the initial $2$ and subsequent string of $1$ entries of length $j$, the lattice walk is at $(j-1,0)$ (since the second $1$ is a down step). Then there is a $2$ and at least $j-1$ entries of value $2'$ following the initial string of $1$s, which bring the path left to the $y$-axis. Then since only the first of the remaining $1^\ast$s can be unprimed, the lattice walk contains no further down steps, and does not return to the $x$-axis. Thus there are at least as many $1^\ast$s in the second row as $2^\ast$s in $T$.
\end{proof}

\begin{lemma}\label{lem:where-3}
In $T$, there is at most one $3^\ast$, and if there is one, then it is unprimed and occurs in the corner of the outer turn. 
\end{lemma}

\begin{proof}
Suppose there are at least two $3^\ast$ entries. By Lemmas~\ref{L_Flap_Row} and~\ref{L_Flap_Box}, the set of $3^\ast$s must form a vertical strip in the right column. By Lemma~\ref{L_Flap_Corner}, and since $T$ is semistandard, the entry in the lower right corner must be a $3$. Then the subword corresponding to entries of $2^\ast$ or $3^\ast$ has the form $2, 3, 3', \ldots, 3', 2, 2', \ldots, 2'$ for some number of $3'$ entries. If there is at least one $3'$, then after the steps for the $2$ and the $3$, the corresponding lattice walk is on the $y$-axis, so the next $2$ gives a right step, and the following $2'$ entries also give right steps. Hence the walk does not return to the $x$-axis, contradicting that $T$ is ballot. Therefore, there can be at most one $3^\ast$ in the filling, and it must be a $3$ in the lower right corner.
\end{proof}

\begin{figure}
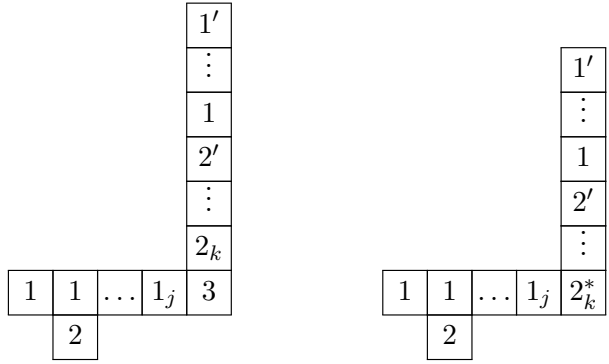

    \centering
    
    $$ \begin{ytableau}  
    \none & \none & \none & \none & 1' \\
    \none & \none & \none & \none & \vdots \\
    \none & \none &  \none & \none & 1 \\
    \none & \none & \none & \none & 2' \\
    \none & \none &  \none & \none & \vdots \\
    \none & \none  & \none & \none & 2_k \\
    1 & 1  & \hdots & 1_j &  3 \\
    \none & 2 \end{ytableau}  \hspace{2cm} 
    \begin{ytableau}  
    \none & \none & \none & \none & \none \\
    \none & \none & \none & \none & 1' \\
    \none & \none & \none & \none & \vdots \\
    \none & \none &  \none & \none & 1 \\
    \none & \none & \none & \none & 2' \\
    \none & \none &  \none & \none & \vdots \\
    1 & 1  & \hdots & 1_j &  2_k^\ast \\
    \none & 2 \end{ytableau}$$
    
    \caption{The forms that any ballot tableau of a frayed ribbon with one turn must take, from Lemmas~\ref{L_Flap_Column} through~\ref{lem:where-3}, where if $k$ is the number of $2^\ast$ entries in the long column and $j$ is the number of $1$ entries in the long row, we have $k<j$.}
    \label{fig:1turn_frayedribbon}
\end{figure}

We now show the above lemmas completely characterize the ballot tableaux of frayed ribbon shapes with one turn. 

\begin{prop}\label{prop:one-turn}
A tableau $T$ of a frayed ribbon shape with one (outer) turn is ballot if and only if it satisfies all of the conditions of Lemmas \ref{L_Flap_Column} through \ref{lem:where-3}.  In particular, they take one of the two forms shown in Figure \ref{fig:1turn_frayedribbon}.
\end{prop}

\begin{proof}
  The forward implication is given by Lemmas \ref{L_Flap_Column} through \ref{lem:where-3}.  Now, suppose a tableau $T$ satisfies the conditions of these six lemmas.  Then it has one of the two forms shown in Figure \ref{fig:1turn_frayedribbon} depending on whether a $3^\ast$ occurs or not (since if it occurs there is exactly one and it appears unprimed in the outer turn).  Then it is easy to check that both of these types of tableaux have ballot reading words.
\end{proof}

Define the \textbf{column height} of a frayed ribbon with one turn to be the number of boxes in the long column above the corner of the outer turn (not including the corner box of the turn itself). Proposition \ref{prop:one-turn} leads to the following explicit Schur $Q$ expansion for one turn frayed ribbons.

\begin{corollary}\label{cor:expansions}
 Let $D$ be a frayed ribbon with one turn of size $n$ and column height $h$.  Then $$Q_D=Q_{(n-1,1)}+2\sum_{i=2}^{m_1(h)} Q_{(n-i,i)}+\sum_{i=2}^{m_2(h)} Q_{(n-i,i,1)}$$ where $m_1(h)=\min(h+1,n-h-2)$ and $m_2(h)=\min(h,n-h-2)$.
\end{corollary}

\begin{proof}
  The first and second terms correspond to the tableaux of the form on the right of Figure \ref{fig:1turn_frayedribbon} (where for $i\ge 2$ the bottom $2$ in the long column may be either primed or unprimed, hence the coefficient of $2$), and the third term corresponds to the tableaux of the form on the left of Figure \ref{fig:1turn_frayedribbon}, which must have an unprimed $2$ just above the $3$ and hence have a coefficient of $1$.
\end{proof}

\begin{theorem}
Suppose $D$ and $E$ are frayed ribbons with one turn of size $n$, where $D$ has column height $h$ and $E$ has column height $\ell$ for $h\neq\ell$. Then $Q_D\neq Q_E$. 
\end{theorem}

\begin{proof}
  Suppose $Q_D=Q_E$.  Then by Corollary \ref{cor:expansions}, we must have $m_1(h)=m_1(\ell)$ and $m_2(h)=m_2(\ell)$.  The former equation states that $\min(h+1,n-h-2)=\min(\ell+1,n-\ell-2)$.  Since $h\neq \ell$, it follows that either $h+1=n-\ell-2$ or $n-h-2=\ell+1$, which are in fact equivalent statements, and so they both hold.
  
  From $m_2(h)=m_2(\ell)$, we have $\min(h,n-h-2)=\min(\ell,n-\ell-2)$, and so by the same reasoning we have $h=n-\ell-2$, which contradicts our equality above.  Hence $Q_D\neq Q_E$.
\end{proof}

\subsection{Frayed ribbons with two turns and column height $0$}\label{sec:height0}

This subsection, along with the next two, will prove the third statement of Theorem \ref{thm:main} in three parts.  As a first step, we will show that all frayed ribbons with two turns and no boxes between the turns have distinct Schur $Q$ functions (up to antipodal reflection).  

Choosing only the representative of each antipodal pair with second-to-last row having length at least three, we see that any such shape is uniquely determined by three parameters, as illustrated in Figure \ref{fig:general_2turn_ribbonbox}: 
\begin{itemize}
    \item The \textbf{top width} $w_1$, defined as the number of squares in the top row,
    \item The \textbf{column height} $h$, defined as the number of squares in the vertical column between the two long rows, and
    \item The \textbf{bottom width} $w_2$, defined as the number of squares in the row second from the bottom.
\end{itemize}

\begin{figure}
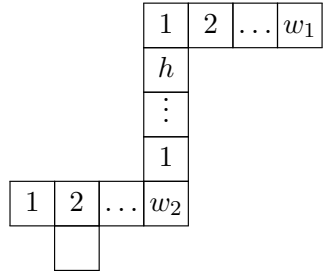

    \centering
    
    $$ \begin{ytableau}  \none & \none & \none & 1 & 2 & \hdots & w_1 \\
    \none & \none &  \none & h \\
    \none & \none &  \none & \vdots \\
    \none & \none  & \none & 1 \\
    1 & 2  & \hdots &  w_2 \\
    \none & \empty \end{ytableau} $$
    
    \caption{A general picture of a frayed ribbon with $2$ turns and parameters $h,w_1,$ and $w_2$. Note that, if $D=\nu/\mu$, then $w_1=\nu_1-\mu_1$ and $w_2 = \nu_{\ell(\mu)+1}$. Further, $n=w_1+h+w_2+1$
    Without loss of generality, by Proposition~\ref{prop:antipodal}, we can assume that $w_2\geq 3$.}
    \label{fig:general_2turn_ribbonbox}
\end{figure}

We now prove two technical lemmas that explicitly calculate a set of Littlewood-Richardson coefficients that arise in the two-turn height $0$ case.

\begin{lemma}\label{lem:value-024}
Let $D=\nu/\mu$ be a frayed ribbon of size $n$ with two turns and parameters $w_1,h,w_2$ where $h=0$.  Then for any $k$ such that $(n-k,k)$ is a shifted partition, the coefficient  $f^{\nu}_{\mu,(n-k,k)}$ of $Q_{(n-k,k)}$ in the expansion of $Q_{D}$ is given by:
$$ f^{\nu}_{(n-k,k),\mu} = \begin{cases} 4 & \text{ if } 2\leq k \leq \min(w_1+1,w_2)-1 \\ 2 & \textrm{ if } k=\min(w_1+1,w_2)  \\ 0 &  \textrm{ if } k \geq \min(w_1+1,w_2)+1  \end{cases}.$$
In particular, when $w_1+1\neq w_2$ there is a unique value of $k$ for which $f^{\nu}_{(n-k,k),\mu}=2$.  When $w_1+1=w_2$, since $(w_1+1,w_2)$ is not a shifted partition, there is no value of $k$ for which $f^{\nu}_{(n-k,k),\mu}=2$.
\end{lemma}

\begin{proof}
  \textbf{Case 1.} Suppose $w_1+1\ge w_2$, so that $\min(w_1+1,w_2)=w_2$.  We count the ballot tableaux of shape $D$ with content $(n-k,k)$ for each $k$.  If $k\geq w_2+1$, then since there must be only $1^\ast$ entries in the top row in any ballot tableau by Lemma \ref{lem:top-row}, the remaining $w_2+1$ boxes must contain $2^\ast$ entries. However, there is no semistandard filling of the three boxes
$$\begin{ytableau} \empty & \empty \\ \none & \empty \end{ytableau}$$
with only $2^\ast$ entries, hence there is no ballot tableau of $D$ with content $(n-k,k)$. So the coefficient is $0$ in this case. 

If $k=w_2$, we first note that $n=w_1+w_2+1$, so the content $(n-k,k) = (w_1+1,w_2)$ is not a valid shifted partition (and hence the second case of the theorem not appearing when $w_1+1=w_2$).  So we may assume in this subcase that $w_1+1>w_2$.  Then there are only two semistandard fillings with this content in canonical form whose top row is all $1^\ast$ entries: 
$$\begin{ytableau} 
\none & \none & \none & \none & 1^\ast & 1 & \hdots & 1 \\ 
1 & 2' & 2 & \hdots & 2 \\ 
\none & 2 
\end{ytableau}$$
where the first $1$ in the final row can be primed or unprimed.  We claim that these fillings are in fact both ballot as well.  Indeed, in the lattice walk for this filling, after reading the second row, the walk is at $(0,w_2-1)$.  Since $w_1+1>w_2$, we have $w_1\ge w_2$, so after the $1^*$, there are at least $w_2-1$ entries with value $1$ in the top row.  This brings the walk back to the $x$-axis, so the fillings are ballot.  It follows that $f^{\nu}_{(n-k,k),\mu}=2$ when $(n-k,k)$ is a valid content in this case.

Finally, suppose $2\leq k\leq w_1+1$.  Then there are four valid fillings:
$$\begin{ytableau} \none &\none & \none & \none & \none & \none  & \none & 1^\ast & 1 & \hdots & 1 \\ 1 & 1 & \hdots & 1 & 2^\ast  & 2 & \hdots & 2 \\ \none & 2 \end{ytableau} $$ 
where the first $2$ in the second row and $1$ in the top row can be primed or unprimed. Since $k\leq w_1+1$, there are at most $w_1\leq w_2$ entries with value $2^\ast$ in the second row, and $w_2$ entries with value $1$ in the final row, so the corresponding lattice walk returns to the $x$-axis. Therefore $f_{(n-k,k),\mu}^{\nu} = 4$ in this case.

\textbf{Case 2.} Suppose that $w_1+1<w_2$, so that $\min(w_1+1,w_2)=w_1+1$ and $w_1+1\neq w_2$. If $k\geq w_1+2$, then for a filling of $D$ to be ballot and semistandard we need all $1^\ast$ entries in the top row, the bottommost square contains a $2$, and the remaining $w_1+1$ of the $2^\ast$ entries occur in the second row.  After reading past these $2^\ast$ entries (all of which but the first must be $2$ by semistandardness), the walk must be at height $w_1+1$, But only $w_1$ of the $1^\ast$ entries follow it in the top row.  Thus the lattice walk cannot return to the $x$-axis, and so there are no ballot tableaux.  Hence in this case $f^{\nu}_{(n-k,k),\mu}=0$.

If $k=w_1+1$, then since $w_1+1<w_2$ is a strict inequality, we find four semistandard fillings of valid content $(w_2,w_1+1)$ in canonical form whose top row contains all $1^\ast$ entries:

$$ \begin{ytableau} \none & \none & \none & \none & \none & \none & \none & 1^\ast & \hdots & 1 \\ 1 & 1 & \hdots & 1 & 2^\ast & 2 & \hdots & 2 \\ \none & 2 \end{ytableau} $$
We will show that only two of these are ballot, namely when the top $1^\ast$ is unprimed.   Indeed, if $x$ is the number of $1$ entries in the second row, then since there are $k-1=w_1$ entries with value $2^\ast$ in the second row, the lattice walk is at $(x-1,w_1)$ after reading this row.  Then the remaining $w_1$ entries on the top row return the walk to the $x$-axis if and only if the $1^\ast$ is unprimed so that its arrow points down.  Hence there are two valid fillings in this case, and $f^{\nu}_{(n-k,k),\mu}=2$.

Finally, if $2\leq k\leq w_1$, then any ballot tableau has at most $w_1-1$ entries with value $2^\ast$ in the second row, and we find four ballot tableaux by the lattice walk reasoning above: $$ \begin{ytableau} \none & \none & \none & \none & \none & \none & \none & 1^\ast & 1 & \hdots & 1 \\ 1 & 1 & \hdots & 1 & 2^\ast & 2 & \hdots & 2 \\ \none & 2 \end{ytableau} $$ Hence there are four ballot tableaux in this case and $f^{\nu}_{(n-k,k),\mu}=4$.
\end{proof}

\begin{lemma}\label{lem:value-012}
Let $D=\nu/\mu$ be a frayed ribbon of size $n$ with two turns and parameters $w_1,h,w_2$ where $h=0$. Then for any $k$ such that $(n-k,k)$ is a shifted partition, the coefficient  $f^{\nu}_{\mu,(n-k-1,k,1)}$ of $Q_{(n-k-1,k,1)}$ in the expansion of $Q_{D}$ is given by:
$$f^{\nu}_{(n-k-1,k,1),\mu} = \begin{cases} 2 \textrm{ if } k\leq \min(w_1,w_2)-1 \\ 1 \textrm{ if } k = \min(w_1,w_2) \\ 0 \textrm{ if } k\geq \min(w_1,w_2)+1 \end{cases} $$
In particular, when $w_1\neq w_2$ there is a unique value of $k$ for which $f^{\nu}_{(n-k-1,k,1),\mu}=1$.  When $w_1=w_2$, since $(w_1,w_2,1)$ is not a shifted partition, there is no value of $k$ for which $f^{\nu}_{(n-k-1,k,1),\mu}=1$.
\end{lemma}

\begin{proof}
  Note that in any ballot tableau with a $3$, since the final row must consist of only $1^\ast$ entries, we must have the $3$ entry in the bottommost square.  Then, if the content is $(n-k-1,k,1)$, all $k$ of the $2^\ast$s are unprimed and in the second row by ballotness, canonical form, and semistandardness.  By semistandardness, the top row is of the form $1^\ast 1\cdots 1$, and so there are at most two ballot tableaux with this content for any $k$.
  
  First suppose that $k\geq \min(w_1,w_2)+1$, so $k\geq w_1+1$ or $k\geq w_2+1$. In the first case, the $1/2$-walk is at height $k=w_1+1$ at the end of the second row, and the top row $1^\ast111\cdots 1$ of length $w_1$ does not contain enough entries to bring the walk back to the $x$-axis, so there are no ballot tableaux. In the second case, there are more $2^\ast$ entries than there are boxes in the second row, so there are no ballot tableaux in this case and we have $f^{\nu}_{(n-k-1,k,1),\mu}=0$.

Next, suppose $k=\min(w_1,w_2)$. If $w_1=w_2$, then $(n-k-1,k,1) = (w_1,w_2,1)$ does not have valid content, so we may assume $w_1\neq w_2$.  If $\min(w_1,w_2)=w_2$, then $k=w_2 < w_1$, so there is one ballot tableau:
$$\begin{ytableau} \none & \none & \none & 1 & 1 & \hdots & 1 \\ 2 & 2 & \hdots & 2 \\ \none & 3 \end{ytableau} $$ (Note that the first $1$ in the top row must be unprimed by the canonical form condition in this case, and it is easily checked that the ballot walk condition holds.)
If instead $\min(w_1,w_2)=w_1$, since $k=w_1<w_2$, there is again one ballot tableau:
$$\begin{ytableau} \none & \none & \none & \none & \none & 1 & \hdots & 1 \\ 1 & \hdots & 1 & 2 & \hdots & 2 \\ \none & 3 \end{ytableau} $$
Indeed, after the second row, the $1/2$-walk is at $(x,w_1)$ for some $x\geq 1$, and so the $w_1$ entries $1^\ast111\cdots 1$ in the top row return the walk to the $x$-axis only when the first $1$ is unprimed. Thus in this case we have $f^{\nu}_{(n-k-1,k,1),\mu}=1$.

Finally, suppose $k\leq \min(w_1,w_2)-1$, so $k\leq w_1-1$ and $k\leq w_2-1$. We find two fillings:
$$\begin{ytableau} \none & \none & \none & \none & \none & 1^\ast & 1 & \hdots & 1 \\ 1 & \hdots & 1 & 2 & \hdots & 2 \\ \none & 3 \end{ytableau} $$
Since $k\leq w_2-1$, the second row consists of a string of at least one $1$ entry, followed by $k$ $2$s, so the $(1,2)-$lattice walk is at $(x,k)$ for $x\geq 1$ and $k\leq w_1-1$. Then the $w_1$ $1^\ast$ entries in the last row return the walk to the $x$-axis, regardless of whether the first $1$ is primed or unprimed. Thus there are $2$ ballot tableaux of this content, and $f^{\nu}_{(n-k-1,k,1),\mu}=2$.
\end{proof}

We can now prove the main result of this section.

\begin{theorem}\label{thm:h1_t2_flap}
Suppose $D$ and $E$ are frayed ribbon shapes of size $n$ with $2$ turns and parameters $w_1,h,w_2$ and $w_1',h',w_2'$ respectively, where $h=h'=0$ and $w_2,w_2'\ge 3$. Then if $w_1\neq w_1'$, we have $Q_D\not= Q_E$. 
\end{theorem}

\begin{proof}
Since $D$ and $E$ both have height $0$, we have $w_2=n-1-w_1$ and $w_2'=n-1-w_1'$.  Thus, since $w_1\neq w_1'$ by assumption, we also have $w_2\neq w_2'$. 

\textbf{Case 1.} Suppose $\min(w_1,w_2)=w_1$ and $\min(w_1',w_2')=w_1'$.   Then by Lemma \ref{lem:value-012} the unique (or nonexistent) value of $k$ at which the coefficient $f^{\nu}_{(n-k,k),\mu}$ is $1$ is either different for $D$ and $E$, or exists for one of $D$ and $E$ and does not exist for the other.  Hence $Q_D\neq Q_E$ in this case.

\textbf{Case 2.} Suppose $\min(w_1,w_2)=w_2$ and $\min(w_1',w_2')=w_2'$.  Then since $w_2\neq w_2'$, the same proof as in Case 1 shows $Q_D\neq Q_E$.

\textbf{Case 3.} Suppose that one of the minima has subscript $1$ and the other $2$; without loss of generality suppose $\min(w_1,w_2)=w_1$ and $\min(w_1',w_2')=w_2'$.  If $w_1\neq w_2'$ then we are done as before; otherwise, suppose $w_1=w_2'$.  Then $\min(w_1'+1,w_2')=w_2'$ and $\min(w_1+1,w_2)$ is either $w_1+1$ or $w_2$.  But $w_1+1\neq w_2'$ since $w_1=w_2'$, and $w_2\neq w_2'$, so in either case we have $$\min(w_1+1,w_2)\neq \min(w_1'+1,w_2').$$ So, by Lemma \ref{lem:value-012}, the unique (or nonexistent) value of $k$ for which the coefficient $f^\nu_{(n-k-1,k,1),\mu}$ is $1$ is either different for $D$ and $E$, or exists for one of $D$ or $E$ and does not exist for the other.  Hence $Q_D\neq Q_E$.
\end{proof}

\subsection{Frayed ribbons with two turns and column height $1$}

We now show that all Schur $Q$ functions of frayed ribbons with column height $1$ are distinct from one another.  We use the same notation $w_1,h,w_2$ established at the start of Section \ref{sec:height0}, and note that throughout this section we will have $h=1$.

\begin{lemma}\label{lem:value-02468}
Let $D=\nu/\mu$ be a two-turn frayed ribbon of size $n$ with parameters $w_1,h,w_2$ where $h=1$.  Then for any $k$ such that $(n-k,k)$ is a shifted partition, the coefficient  $f^{\nu}_{\mu,(n-k,k)}$ of $Q_{(n-k,k)}$ in the expansion of $Q_{D}$ is given by:
$$f^{\nu}_{(n-k,k),\mu} = \begin{cases} 8\textrm{ if } 3\leq k \leq \min(w_1,w_2-1) \\    6\textrm{ if } k=\min(w_1,w_2-1)+1 \textrm{ and } w_1\neq w_2-1 \\   4\textrm{ if } k=\min(w_1,w_2-1)+1\textrm{ and } w_1=w_2-1 \\   2 \textrm{ if } k=\min(w_1,w_2-1)+2 \\   0 \textrm{ if } k\geq \min(w_1,w_2-1)+3 \end{cases}$$

In particular, there is a unique value of $k$ for which $f^{\nu}_{(n-k,k),\mu}$ is either $4$ or $6$.
\end{lemma}

\begin{proof}
  In order for a filling with content $(n-k,k)$ to be semistandard and ballot, the upper row must consist of only $1^\ast$ entries (Lemma \ref{lem:top-row}), and the bottommost square must contain a $2$. 

First, suppose that $3\leq k\leq \min(w_1,w_2-1)$.
Since $k\geq 3$, there are $8$ possible fillings: $$ \begin{ytableau} 
\none & \none & \none & \none & \none & \none & 1' & \hdots & 1 \\ 
\none & \none & \none & \none & \none & \none & 1^\ast \\ 
1 & \hdots & 1 & 2^\ast & 2 & \hdots & 2 \\ 
\none & 2 \end{ytableau} \hspace{1cm} 
\begin{ytableau} 
\none & \none & \none & \none & \none & \none & 1^\ast & \hdots & 1 \\ 
\none & \none & \none & \none & \none & \none & 2' \\ 
1 & \hdots & 1 & 1  & 2^\ast & \hdots & 2 \\ 
\none & 2 \end{ytableau} $$
Since $k\leq w_2-1$, the third row contains at least two $1$ entries, so after the string of $2^\ast$ entries in the reading word, the $1/2$-walk is at either $(x,k-1)$ or $(x-1,k-2)$ for $x\geq 1$, according to the two cases shown above. Then, since $k\leq w_1$, there are at at least $k-1$ entries with value $1$ to return the walk to the $x$-axis. Thus each of these fillings is ballot.

Next, suppose that $k=\min(w_1,w_2-1)+1$.  Note that since $w_1\ge 2$ and $w_2\ge 3$ by the definition of the two-turn shape, we have $k\ge 3$ again.  If $w_1\neq w_2-1$, and $k=w_1+1$, we claim  there are only $6$ ballot tableaux, as the type on the left above must have an unprimed $1$: 
$$ \begin{ytableau} 
\none & \none & \none & \none & \none & \none & 1' & \hdots & 1 \\ 
\none & \none & \none & \none & \none & \none & 1 \\ 
1 & \hdots & 1 & 2^\ast & 2 & \hdots & 2 \\ 
\none & 2 \end{ytableau} \hspace{1cm} 
\begin{ytableau} 
\none & \none & \none & \none & \none & \none & 1^\ast & \hdots & 1 \\ 
\none & \none & \none & \none & \none & \none & 2' \\ 
1 & \hdots & 1 & 1  & 2^\ast & \hdots & 2 \\ 
\none & 2 \end{ytableau} $$
Indeed, since $k<w_2$, we get $k\leq w_2-1$, so there are at least two $1$ entries in the third row in the left hand diagram above, and at least three in the right diagram.  Thus the $1/2$-walk is at either $(x,k-1)$ for some $x\ge 1$ for the left hand diagram, or $(x-1,k-2)$ for some $x\ge 2$ respectively after the string of $2^\ast$ entries. Since $k=w_1+1$, then if the string of $2^\ast$ entries ends in the third row, the following $1$ must be unprimed so that there are $k-1$ unprimed $1$ entries to return the subwalk to the $x$-axis. 

If $k=w_2=\min(w_1,w_2-1)+1$ with $w_1\neq w_2-1$, then we claim there are again $6$ ballot tableaux: 
$$ \begin{ytableau} \none  & \none & \none & \none & 1' & \hdots & 1 \\ \none & \none  & \none &\none & 1^\ast \\ 1 & 2' & 2 & \hdots & 2 \\ \none & 2   \end{ytableau} \hspace{1cm} \begin{ytableau} \none & \none & \none & \none &  1^\ast & \hdots & 1 \\  \none & \none & \none &\none & 2' \\ 1 & 1 & 2^\ast & \hdots & 2 \\ \none & 2  \end{ytableau} $$
Indeed, in the two possible forms of tableaux above, after the string of $2^\ast$ entries, the walk is at $(0,k-1)$ or $(0,k-2)$ respectively.  It is followed by a $1^\ast$ entry that moves the walk one step to the right and then at least $w_1-1\ge k-1$ entries with value $1$ to return the walk to the $x$ axis (since $k<w_1+1$).  Furthermore, since $k=w_2$, if the single box in the second row contains a $1^\ast$, then the first $2^\ast$ entry in the third row is directly above the $2$ in the bottommost box, and so it must be primed for the tableau to be semistandard. Hence there are exactly $6$ ballot tableaux.

If $k=w_1+1=w_2$, then we claim the only $4$ ballot tableaux are of the form: 
$$ \begin{ytableau} \none & \none & \none & \none &  1^\ast & \hdots & 1 \\  \none & \none & \none &\none & 2' \\ 1 & 1 & 2^\ast & \hdots & 2 \\ \none & 2  \end{ytableau} $$
We first show there cannot be a $1^\ast$ in the single box in the second row.  If so, since $k=w_2$ the first $2^\ast$ in the third row must be primed and is above the bottommost $2$. Then the walk of the word is at $(0,k-1)$ following the string of $2^\ast$ entries. Since $k=w_1+1$, the remaining string $1^\ast,1',1,\ldots,1$ of length $w_1+1$ does not return the walk to the $x$-axis.  Thus there is a $2'$ in the second row.  In this case, there are two $1$ entries in the third row, so the subwalk is at $(0,k-2)$ after the string of $2^\ast$ entries. It follows that the remaining single $1^\ast$ and $w_1-1 = k-2$ more entries with value $1$ return the walk to the $x$-axis. Hence there are exactly $4$ ballot tableaux in this case.

Now suppose that $k=\min(w_1,w_2-1)+2$. Note that, since $n=w_1+w_2+2$, if $|w_1-(w_2-1)|\leq 1$, the content $(n-k,k)$ is not valid. So $|w_1-(w_2-1)|\geq 2$. 

If $k=w_2+1$, then there are only $2$ possible ballot tableaux: 
$$\begin{ytableau} \none & \none & \none & \none &  1^\ast & \hdots & 1 \\  \none & \none & \none &\none & 2' \\ 1 & 2' & 2 & \hdots & 2 \\ \none & 2  \end{ytableau}$$
To see that these are both ballot, after the string of $2^\ast$ entries, the $1/2$-walk is at $(0,k-1)$, and the next $1^\ast$ moves the walk right one step to $(1,k-1)$. From above, we get $k-1= w_2\leq w_1-1$, so the remaining string of at least $k-1$ entries of value $1$ returns the walk to the $x$-axis. 

If instead $k=w_1+2$, then there are also $2$ ballot tableaux:
$$ \begin{ytableau} \none & \none & \none & \none & \none & 1 & \hdots & 1 \\ \none & \none & \none & \none &\none & 2' \\ 1 & \hdots & 1 & 2^\ast & \hdots & 2 \\ \none & 2   \end{ytableau} $$
In particular, by the inequality $|w_1-(w_2-1)|\geq 2$, since $w_1<w_2-1$ in this case we have $k-1=w_1+1\le w_2-2$, and so there are at least two $1$ entries in the third row. Thus, if the single box in the second row contains a $1^\ast$, then after the string of $2^\ast$ entries, the walk is at $(x,k-1)$ for some $x\geq 1$.  Then there are at most $w_1=k-2$ unprimed $1$ entries remaining, so the walk does not return to the $x$-axis.  It follows that the second row must contain a $2'$. In this case, after the string of $2^\ast$ entries, the walk is at $(x-1,k-2)$. Then the remaining string $1^\ast1\cdots 1$ of length $w_1=k-2$ returns the walk to the $x$-axis only if the first $1$ is unprimed. So there are two ballot tableaux with this content. 

Finally, suppose that $k\geq \min(w_1,w_2-1)+3$. If $k\geq w_1+3$, then the final string of $2^\ast$ entries has length at least $w_1+2$, but there are at most $w_1+1$ entries of value $1^\ast$ remaining, so the walk does not return to the $x$-axis. If $k\geq w_2+2$, then there are no semistandard fillings of $D$, since the top row and first box of the third row must contain $1^\ast$ entries. In either case, there are no ballot tableaux with this content.
\end{proof}

\begin{lemma}\label{lem:value-024-h1}
Let $D=\nu/\mu$ be a two-turn frayed ribbon of size $n$ with parameters $w_1,h,w_2$ where $h=1$.  Then for any $k$ such that $(n-k-2,k,2)$ is a shifted partition, the coefficient  $f^{\nu}_{\mu,(n-k-2,k,2)}$ of $Q_{(n-k-2,k,2)}$ in the expansion of $Q_{D}$ is given by:
$$f^{\nu}_{(n-k-2,k,2),\mu} = \begin{cases} 4 \textrm{ if } k \leq \min(w_1,w_2)-1 \\ 2 \textrm{ if } k = \min(w_1,w_2) \\ 0 \textrm{ if } k\geq\min(w_1,w_2)+1 \end{cases}$$
In particular, when $w_1\neq w_2$ there is a unique value of $k$ for which $f^{\nu}_{\mu,(n-k-2,k,2)}=2$.  When $w_1=w_2$, since $(w_1,w_2,2)$ is not a shifted partition, there is no value of $k$ for which $f^{\nu}_{\mu,(n-k-2,k,2)}=2$.
\end{lemma}

\begin{proof}
   For a filling with content $(n-k-2,k,2)$ to be semistandard and ballot, the first row must consist only of $1^\ast$ entries (Lemma \ref{lem:top-row}), the single box in the second row must contain a $2$ for the $2/3$-walk to end on the $x$ axis, the third row has reading word $1\cdots 12 \cdots 23^\ast$, and the fourth row must contain a $3$. Further, we must have $k\geq 3$, so that $\lambda$ is a strict partition. 

First, suppose that $k \leq \min(w_1,w_2)-1$. There are $4$ semistandard tableaux satisfying the conditions above: 
$$ \begin{ytableau} \none & \none & \none & \none & \none & \none & 1^\ast & \hdots & 1 \\ \none & \none & \none & \none & \none &\none & 2 \\ 1 & \hdots & 1 & 2 & \hdots & 2 & 3^\ast \\ \none & 3   \end{ytableau} $$
We show that these four tableaux are all ballot.  Since $k\geq 3$, the $2/3$-walk returns to the $x$-axis. Since $k\leq w_2-1$, there is at least one $1$ entry in the third row, so following the final $2$ in the reading word, the $1/2$-walk is at $(x,k)$ for some $x\geq 1$. Since $k\leq w_1-1$, there are at least $k$ unprimed $1$ entries in the first row, so the $1/2$-walk returns to the $x$-axis. Hence there are $4$ ballot tableaux with this content. 

Next, suppose that $3\leq k = \min(w_1,w_2)$. If $w_1=w_2$, then $n-k-2=n-w_1-2=w_2=k$, so $(n-k-2,k,2)$ is not a valid content. If $k=w_1<w_2$, then we get two fillings: 
$$ \begin{ytableau} \none & \none & \none & \none & \none & \none & 1 & \hdots & 1 \\ \none & \none & \none & \none & \none &\none & 2 \\ 1 & \hdots & 1 & 2 & \hdots & 2 & 3^\ast \\ \none & 3   \end{ytableau} $$
Since $k<w_2$, there is at least one $1$ entry in the third row, so after the $2$ in the second row, the $1/2$-walk is at $(x,k)$ for some $x\geq 1$. Then, since $k=w_1$, the top row must consist of only unprimed $1$ entries for the subwalk to return to the $x$-axis. Thus there are $2$ valid fillings.

If instead $k=w_2<w_1$, then there are two such fillings: 
$$ \begin{ytableau}  \none & \none & \none & 1 & \hdots & 1 \\  \none & \none &\none & 2 \\  2 & \hdots & 2 & 3^\ast \\ \none & 3   \end{ytableau} $$
Since $k=w_2$, the third row contains no $1$ entries, so the first $1$ entry in the reading word is in the first row, and must be unprimed for the tableau to be in canonical form. Following the string of $2$ entries, the $1/2$-walk is at $(0,k)$. Since $k<w_1$, there are enough $1$ entries following the $2$ entries to return the walk to the $x$-axis. Therefore there are $2$ valid fillings in this case.

Finally, suppose that $k\geq \min(w_1,w_2)+1$. If $\min(w_1,w_2)=w_1$ so that $k\geq w_1+1$, then after the last $2$, the $1/2$-walk is at $(x,k)$ for some $x\geq 0$. Then the remaining $w_1<k$ entries of $1$ do not return the walk to the $x$-axis. If $\min(w_2,w_2)=w_2$ so that $k\geq w_2+1$, then since the $3$ entries must be in the third and fourth rows, there is a $2$ entry in the top row, which cannot happen by Lemma \ref{lem:top-row}.  So there are no valid fillings with this content.
\end{proof}

We can now prove distinctness for height $1$ frayed ribbons with two turns. 

\begin{theorem}\label{thm:h2_t2_flap}
If $D$ and $E$ are frayed ribbon shapes of size $n$ with $2$ turns, column height $1$, and different first row width $w_1$, then $Q_D\neq Q_E$.
\end{theorem}

\begin{proof}
Suppose that $Q_D=Q_E$.  Let $w_1,h=1,w_2$ be $D$'s parameters and $w_1',h=1,w_2'$ be $E$'s parameters.  Then by Lemma \ref{lem:value-024-h1}, among the terms of the form $Q_{(n-k-2,k,2)}$ in $Q_D$'s expansion, there is either a unique coefficient of $2$ or none.  If there is a unique one then it occurs at $k=\min(w_1,w_2)$ and similarly for $Q_E$, so since $w_1\neq w_1'$ we must have $w_1=w_2'$ (and $w_2=w_1'$).

But then, by Lemma \ref{lem:value-02468}, there is a unique value of $k$ for which the coefficient of $Q_{(n-k,k)}$ in $Q_D$ is either $4$ or $6$, namely, $k=\min(w_1,w_2-1)+1$, and similarly for $E$.  Thus $w_1=w_2'-1$, a contradiction.
\end{proof}

\subsection{Distinguishing height $0$ from height $1$}

In the previous two sections we showed that all height $0$ two-turn frayed ribbons have distinct Schur $Q$ functions, and that all height $1$ such shapes also have distinct Schur $Q$ functions.  We now show that these two classes are also all distinct from each other.  To do so, we in fact prove a more general statement - that any shape with height $0$ actually has a Schur $Q$ function that is distinct from all other frayed ribbons.

\begin{prop}
If $D$ is a frayed ribbon shape of size $n$ with two turns and column height $0$, then $Q_D\neq Q_E$ for any frayed ribbon shape $E\neq D$, other than its antipodal $D^a$.
\end{prop}

\begin{proof}
  We already know that if $E$ does not have two turns or if it has two turns and column height $0$, $Q_D\neq Q_E$ from Theorem \ref{thm:h1_t2_flap} and Corollary \ref{cor:turns}.  Thus it suffices to consider the case in which $E$ has $2$ turns and height greater than $0$.  We may assume without loss of generality that $E$'s second-to-bottom row has more than $2$ entries by Proposition \ref{prop:antipodal}, and simply show that $Q_D\neq Q_E$ since $E$ will not be the antipodal shape $D^a$ under this assumption.
  
  Consider the coefficient of $Q_{(n-3,2,1)}$ in the straight shape Schur $Q$ expansion of both $Q_D$ and $Q_E$. For $D$, this coefficient is either $1$ or $2$ by Lemma \ref{lem:value-012}.
  
  For $E$, on the other hand, we show that this coefficient is at least $3$.  Indeed, there are always two Littlewood-Richardson tableaux of the form shown below at left, in which the $3$ is in the corner, one $2$ is in the bottommost square, and the other $2$ is above the $3$ (and unprimed so that the $2/3$ word is ballot).  The $1$ at the top of the column can be either primed or unprimed and the $1/2$ word will still be ballot, so this gives two possibilities.  Then, there is also always at least one of the form shown below at right, in which a $3$ is in the bottommost square and the $2$'s are in the next to last row; if the $1$ in the column is unprimed then this guarantees that the $1/2$ word is ballot since there is at least one $1$ before the $2$s and at least two $1$s after them to bring the walk back down to the $x$-axis.
  $$\begin{ytableau}
    \none & \none & \none & 1^\ast & 1 & 1 \\
    \none & \none & \none & 2 \\
    1 & 1 & 1 & 3 \\
    \none & 2
  \end{ytableau}\hspace{2cm}
  \begin{ytableau}
    \none & \none & \none & 1' & 1 & 1 \\
    \none & \none & \none & 1 \\
    1 & 1 & 2 & 2 \\
    \none & 3
  \end{ytableau}
  $$
  
  Thus the coefficient of $Q_{(n-3),2,1}$ in $Q_D$ is at most $2$, and its coefficient in $Q_E$ is at least $3$, and therefore $Q_D\neq Q_E$ as desired.
\end{proof}

From this proof and Theorems \ref{thm:h1_t2_flap} and \ref{thm:h2_t2_flap}, we can conclude that the frayed ribbon shapes with two turns and height either $1$ or $2$ give another collection of distinct skew Schur $Q$-functions:

\begin{corollary}
 The skew Schur $Q$ functions $Q_D$, where $D$ ranges over all two-turn frayed ribbon shapes of height $1$ or $2$, are all distinct.
\end{corollary}

\section{Further Observations}\label{sec:conclusion}

We conclude here with several examples and observations pertaining to natural generalizations of Conjecture \ref{conj:frayed} and Theorem \ref{thm:main}.  All of the examples in this section were found using Sage \cite{sagemath}.

\subsection{Equality with two boxes on the staircase}

It is natural to ask whether the ``frayed'' aspect of frayed ribbons is in fact enough to distinguish any Schur $Q$ functions that are not antipodal.  More specifically, perhaps any two distinct non-antipodal connected skew shifted shapes having at least two boxes on the staircase diagonal have distinct Schur $Q$ functions.  The following counterexample shows that this generalization does not hold.

\begin{example}
 There exists a non-antipodal pair of shapes having at least two boxes on the staircase which have equal Schur $Q$ functions (but are not frayed ribbons).   The smallest such pair has size $8$ and is shown below.
$$\begin{ytableau}
 \none & \none & \empty \\
  \none & \none & \empty \\
  \empty & \empty & \empty \\
 \empty & \empty \\
 \none & \empty 
\end{ytableau}\hspace{2cm} \hspace{0.5cm}
\begin{ytableau}
  \none & \none & \none & \empty \\
  \empty & \empty & \empty & \empty \\
 \empty & \empty \\
 \none & \empty 
\end{ytableau}$$
Their Schur-$Q$ expansions  are $$Q_{(6, 5, 4, 2, 1)/(5, 4, 1)}=Q_{(6, 5, 2, 1)/(5, 1)}=Q_{(6,2)}+2Q_{(5,3)}+2Q_{(5,2,1)}+2Q_{(4,3,1)}.$$
\end{example}

\subsection{A frayed ribbon and a non-frayed near-ribbon}

The following example shows that the Schur $Q$ functions of frayed ribbons are not necessarily distinct from those of other near-ribbons that are not frayed.

\begin{example}
We have $$Q_{(4,3,1)/(3)}=Q_{(4,3)/(2)}=Q_{(4,1)}+Q_{(3,2)},$$ and the shapes $(4,3,1)/(3)$ and $(4,3)/(2)$ are a frayed ribbon and near-ribbon respectively:

$$\begin{ytableau}
 \none & \none & \none & \empty \\
  \none & \empty & \empty & \empty \\
  \none & \none & \empty & \none 
\end{ytableau}\hspace{2cm} \hspace{0.5cm}
\begin{ytableau}
 \none & \none & \empty & \empty \\
  \none & \empty & \empty & \empty 
\end{ytableau}$$
\end{example}

\subsection{Equality of non-frayed near-ribbons}

There exist pairs of near-ribbons that are \textit{both} not frayed in which equality holds, and their Schur $Q$ functions are not trivially equal by being antipodal, transposed, or antipodal transposed shapes.  An example is given below.
\begin{example}
  We have $$Q_{(7,6,5,3)/(6,5,2)}=Q_{(7,6,5,1)/(6,4,1)}=3Q_{(4, 3, 1)} + 3Q_{(5, 2, 1)} + 5Q_{(5,3)} + 4Q_{(6,2)} + Q_{(7, 1)},$$ and the shapes $(7,6,5,3)/(6,5,2)$ and $(7,6,5,1)/(6,4,1)$ are non-frayed near-ribbons that are not equivalent under any combination of the antipodal and transpose operations:
$$\begin{ytableau}
  \none & \none & \none & \empty \\
  \none & \none & \none & \empty \\
  \none & \empty & \empty & \empty \\
  \empty & \empty & \empty
\end{ytableau}\hspace{2.5cm}
\begin{ytableau}
  \none & \none & \none & \empty \\
  \none & \none & \empty & \empty \\
  \empty & \empty & \empty & \empty \\
  \empty & \none & \none & \none  
\end{ytableau}$$
\end{example}

\subsection{Schur $Q$-positive differences with ribbons}

Our results also shed light on the more general problem of determining when the difference of two skew Schur $Q$ functions is Schur $Q$-positive.  In particular, we have the following observation.

\begin{lemma}\label{lem:pos-diff}
Let $D$ be a shifted skew shape, and suppose $E$ is formed from $D$ by moving the topmost $k$ rows each one unit to the right for some $k$.  Then $Q_E-Q_D$ is Schur $Q$-positive.
\end{lemma}

\begin{proof}
  Let $T$ be a ballot tableau of shape $D$.  Then if we shift the topmost $k$ rows each one unit to the right in $T$, we obtain a tableau $T'$ of shape $E$.  We claim $T'$ is ballot.  Its reading word is the same as that of $T$, so it still has the ballot reading word and canonial form conditions.  
  
  To see that it is semistandard, certainly it is still semistandard across each row, and for semistandardness in columns we only need to check the adjacent squares between the $k$-th row from the top and the $k+1$st.  If $i^\ast$ is a label in the $k$-th row in $T$, then it is less than (or possibly equal to, if primed) the entry $j^\ast$ below it in $T$, which in turn is less than (or equal to, if unprimed) the entry $r^\ast$ to the right of $j$ in $T$, if such an entry exists.  In $T'$, $i^\ast$ is now above $r^\ast$, and by transitivity, the semistandard condition is still satisfied.
  
  It follows that the Schur $Q$ expansion of $Q_D$ has all strictly smaller coefficients than those of $Q_E$, and so $Q_E-Q_D$ is Schur $Q$-positive.
\end{proof}

With Lemma \ref{lem:pos-diff} in mind, we can compare the Schur $Q$ expansions of frayed ribbons to those of ordinary ribbons, by moving all but the bottommost row of a frayed ribbon $D$ one step to the right to form a ribbon $R$.  (We can alternatively think of this as moving the bottommost box of $D$ one step to the left.)  Then we have that $Q_R-Q_D$ is Schur $Q$ positive by Lemma \ref{lem:pos-diff}.  

In \cite{BarekatVanWilligenburg}, Barekat and van Willigenburg conjecture a necessary and sufficient condition for when two ribbon Schur $Q$ functions $Q_R$ and $Q_{R'}$ are equal.  Given two such equal functions $Q_R=Q_{R'}$, if Conjecture \ref{conj:frayed} holds, we would then be able to conclude that the differences $Q_R-Q_D$ and $Q_{R'}-Q_{D'}$ are \textit{distinct} Schur $Q$-positive symmetric functions, where $D,D'$ are the frayed ribbons that map to $R,R'$ respectively under shifting the bottom box to the left.

We give an example of this phenomenon below.

\begin{example}
Using the tools of \cite{BarekatVanWilligenburg}, the pair of ribbons $R,R'$ shown below have equal Schur $Q$ functions:
$$\begin{ytableau}
  \none & \none & \none & \none & \empty & \empty \\
  \none & \none & \none & \none & \empty \\
  \none & \none & \empty & \empty & \empty \\
  \empty & \empty & \empty \\
   \empty & \none
\end{ytableau}\hspace{2.5cm}
\begin{ytableau}
   \none & \none & \none & \empty & \empty \\
   \none & \none & \none & \empty \\
   \none & \none & \empty & \empty \\
   \none & \none & \empty \\
  \empty & \empty & \empty \\
   \empty & \none
\end{ytableau}$$
(These may be generated as the compositions $(1,1)\bullet (1,3,1)$ and $(2)\bullet (1,3,1)$ in the notation of \cite{BarekatVanWilligenburg}).  
However, the corresponding frayed ribbons $D,D'$ are shown below:
$$\begin{ytableau}
  \none & \none & \none & \none & \empty & \empty \\
  \none & \none & \none & \none & \empty \\
  \none & \none & \empty & \empty & \empty \\
  \empty & \empty & \empty \\
  \none & \empty 
\end{ytableau}\hspace{2.5cm}
\begin{ytableau}
   \none & \none & \none & \empty & \empty \\
   \none & \none & \none & \empty \\
   \none & \none & \empty & \empty \\
   \none & \none & \empty \\
  \empty & \empty & \empty \\
  \none & \empty 
\end{ytableau}$$
and do not have equal Schur $Q$ functions.  In this example, we have \begin{align*}
    Q_R-Q_D & =Q_{(4, 3, 2, 1)} + 8\cdot Q_{(5, 3, 2)} + 8\cdot Q_{(5, 4, 1)} + 18\cdot Q_{(6, 3, 1)} + 15\cdot Q_{(6, 4)} \\ & \qquad + 11\cdot Q_{(7, 2, 1)} + 21\cdot Q_{(7, 3)} + 13\cdot Q_{(8, 2)} + 4\cdot Q_{(9, 1)} + Q_{(10)}
\end{align*} and \begin{align*}
    Q_{R'}-Q_{D'} & =Q_{(4, 3, 2, 1)} + 10\cdot Q_{(5, 3, 2)} + 11\cdot Q_{(5, 4, 1)} + 20\cdot Q_{(6, 3, 1)} + 17\cdot Q_{(6, 4)} \\ & \qquad + 11\cdot Q_{(7, 2, 1)} + 21\cdot Q_{(7, 3)} + 13\cdot Q_{(8, 2)} + 4\cdot Q_{(9, 1)} + Q_{(10)}
\end{align*} where \begin{align*}
Q_R=Q_{R'}&=4\cdot Q_{(4, 3, 2, 1)} + 34\cdot Q_{(5, 3, 2)} + 34\cdot Q_{(5, 4, 1)} + 56\cdot Q_{(6, 3, 1)} + 45\cdot Q_{(6, 4)}\\ &\qquad + 24\cdot Q_{(7, 2, 1)} + 45\cdot Q_{(7, 3)} + 21\cdot Q_{(8, 2)} + 5\cdot Q_{(9, 1)} + Q_{(10)}.
\end{align*}
\end{example}

Notice that, in the examples above, the coefficients in the differences $Q_R-Q_D$ are significantly smaller, and potentially easier to get a handle on combinatorially, than those of $Q_R$ itself.  These observations give rise to the following natural problem.

\begin{question}
 Let $R$ be a ribbon having exactly one square in the bottom row and more than one in the second to bottom row.  Let $D$ be the frayed ribbons formed by moving their bottom square one unit to the right.  What is the (positive) Schur $Q$ expansion $Q_D-Q_R$?
\end{question}

 Such an understanding of these expansions would also give a partial resolution to Conjecture \ref{conj:frayed} when the corresponding ribbons have equal Schur $Q$ functions.

\subsection{More Schur $Q$-positive differences and an inductive approach}

Finally, an understanding of the differences that arise when shifting rows by one step could also lead to an inductive approach for proving Conjecture \ref{conj:frayed}.  In particular, given a frayed ribbon $D$ with its last turn being an outer turn, define $D\cdot (r)$ to be the shape formed by adding a row of length $r$ to the upper right of shape $D$.  (See Figure \ref{fig:add-row}.)
\begin{figure}
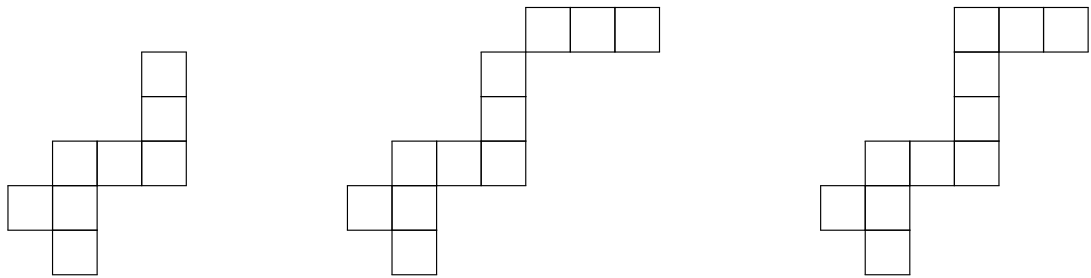

    \centering
    \begin{ytableau}
       \none & \none & \none & \none  \\
       \none & \none & \none & \empty \\
       \none & \none & \none & \empty \\
       \none & \empty & \empty & \empty \\
       \empty & \empty \\
       \none & \empty 
    \end{ytableau}\hspace{2cm}
    \begin{ytableau}
       \none & \none & \none & \none & \empty & \empty & \empty \\
       \none & \none & \none & \empty \\
       \none & \none & \none & \empty \\
       \none & \empty & \empty & \empty \\
       \empty & \empty \\
       \none & \empty 
    \end{ytableau}\hspace{2cm}
    \begin{ytableau}
       \none & \none & \none & \empty & \empty & \empty \\
       \none & \none & \none & \empty \\
       \none & \none & \none & \empty \\
       \none & \empty & \empty & \empty \\
       \empty & \empty \\
       \none & \empty 
    \end{ytableau}
    \caption{If $D$ is the frayed ribbon at left, then $D\cdot (3)$ is pictured at middle, and $\overline{D\cdot (3)}$ is shown at right.}
    \label{fig:add-row}
\end{figure}
Then since the shape is disconnected, $$Q_{D\cdot (r)}=Q_D\cdot Q_{(r)}.$$  Given any other frayed ribbon $E\neq D,D^a$, we similarly have $Q_{E\cdot (r)}=Q_E\cdot Q_r$.   Thus $Q_D\neq Q_E$ if and only if \begin{equation}\label{eq:dot}
    Q_{D\cdot (r)}\neq Q_{E\cdot(r)}
\end{equation} for any $r$.  

Now, define $\overline{D\cdot(r)}$ to be the shape formed by shifting the top row of $D\cdot(r)$ one step to the left (see Figure \ref{fig:add-row}.  By Lemma \ref{lem:pos-diff} we have that $Q_{D\cdot (r)}-Q_{\overline{D\cdot(r)}}$ is Schur $Q$-positive.  A sufficient understanding of these differences could then lead to a proof that $Q_{\overline{E\cdot(r)}}\neq Q_{\overline{D\cdot(r)}}$ starting from the inequality (\ref{eq:dot}), if we know that $Q_D\neq Q_E$.    We therefore state our final question as follows.

\begin{question}
  What is the (positive) Schur $Q$ expansion of $Q_{D\cdot (r)}-Q_{\overline{D\cdot(r)}}$ for a frayed ribbon $D$ and positive integer $r$?
\end{question}

Such an understanding may aid in a proof of Conjecture \ref{conj:frayed} by induction on the number of turns.
\bibliography{refs}
\bibliographystyle{plain}

\end{document}